\documentclass[11pt]{article}
\usepackage[square,authoryear]{natbib}
\usepackage{marticle}

\usepackage{bm}

\usepackage[all]{xy}
\usepackage{amssymb,amsfonts,appendix}
\usepackage{stackengine}
\usepackage{scalerel,stackrel}

\newtheoremstyle{obs}
{3pt}
{3pt}
{}
{}
{\bfseries}
{.}
{.5em}
{}
\theoremstyle{obs}
\newtheorem{remark}[theorem]{Remark}
\newtheorem{example}[theorem]{Example}

\usepackage{authblk}
\newcommand{\lvec}[1]{\overleftarrow{#1}}
\newcommand{\rvec}[1]{\overrightarrow{#1}}
\newcommand{\im}{\textrm{Im}}

\newcommand{\lcf}{\lbrack\! \lbrack}
\newcommand{\rcf}{\rbrack\! \rbrack}

\newcommand{\lag}{\mathcal L}
\begin{document}
	\title{Retraction maps: A seed of geometric integrators. \\
 Part II: Symmetry and reduction}

 \author[1]{Mar\'ia\ Barbero-Li\~n\'an}
 \author[2]{Juan Carlos Marrero}
\author[3]{David\ Mart\'{\i}n de Diego}

\affil[1]{\small Departamento de Matem\'atica Aplicada, Universidad Polit\'ecnica de Madrid, Av. Juan de Herrera 4, 28040 Madrid, Spain. }
\affil[2]{\small ULL-CSIC Geometría Diferencial y Mecánica Geométrica, Departamento de Matemáticas, Estadística e Investigación Operativa and Instituto de Matemáticas y Aplicaciones (IMAULL), University of La Laguna, San Cristobal de La Laguna, Spain.}
\affil[3]{\small Instituto de Ciencias Matem\'aticas (CSIC-UAM-UC3M-UCM), C/Nicol\'as Cabrera 13-15, 28049 Madrid, Spain.}

	\date{\today}
	
	\maketitle

\abstract{In this paper we use  retraction and discretization maps (see \cite{21MBLDMdD}) as a tool for deriving in  a systematic way numerical integrators preserving geometric structures (such as symplecticity or Lie-Poisson structure), as well as methods that preserve symmetry and the associated discrete momentum map. The classical notion of a retraction map leads to the notion of discretization map extended here to the Lie algebroid of a Lie groupoid so that the configuration manifold is discretized, instead of the equations of motion. As a consequence, geometric integrators are obtained preserving the Lie-Poisson structure of the corresponding reduced system. 

   \vspace{2mm}


    \vspace{2mm}
    
}

\tableofcontents

\section{Introduction}

The motivation for this paper lies in the significant role that symmetry plays in the study of dynamical systems, both in their continuous and discrete forms \cite{Marsden-Ratiu,marsden-west}. Typically, this involves the study of mechanics in various phase spaces, each with its own geometric preservation properties (such as symplecticity, Poisson preservation, etc.).  Examples include the Euler-Lagrange equations on the tangent bundle 
 $TQ$, the Euler-Poincaré equations defined on the Lie algebra ${\mathfrak g}$ of a Lie group $G$, Lagrange-Poincaré equations on the quotient space $TQ/G$ (the Atiyah bundle), etc.(see~\cite{CMR01a}), and their corresponding hamiltonian versions on the dual spaces $T^*Q$, ${\mathfrak g}^*$ and $T^*Q/G$, respectively \cite{2003CendraHPeq}. 
In an effort to treat all these different dynamical systems within the same framework, we can work in the setting of Lagrangian mechanics on Lie algebroids or Hamiltonian mechanics with respect to linear Poisson structures \cite{weinstein96,martinez01,LeMaMa} gaining a unified mathematical perspective. The corresponding discrete counterpart is formulated in terms of Lie groupoids. For the aforementioned cases, the discrete equations are defined on $Q\times Q$, $G$ and $(Q\times Q)/G$ (see \cite{MMM06Grupoides,2015JCDavidLocalDiscrete} for discrete Lagrangian mechanics on Lie groupoids).

In particular, in this paper we show how to geometrically derive geometric integrators (\cite{serna,McLachlan-Quispel,hairer,casas-blanes}) preserving Poisson structure  for reduced systems using an appropriate and new notion of retraction (\cite{AbMaSeBookRetraction,Shub}) and discretization maps and the corresponding lift to tangent and cotangent bundles. The discretization maps introduced in~\cite{21MBLDMdD} are extended here to the more general setting of Lie groupoids where numerical integrators for reduced systems are defined. Additionally, our paper is related with geometric integrators which preserve constants of the motion (momentum map) as a consequence of invariance by a Lie group of symmetries. We explored this topic in  a previous paper~\cite{24IFAC}.

In order to extend the constructions in~\cite{21MBLDMdD} we introduce an equivalent definition of the symplectic integrators defined in that paper. 

\subsection{Symplectic integrators from discretization maps}

In \cite{21MBLDMdD} we have used some constructions of geometric mechanics (tangent and cotangent lifts, canonical sym\-plec\-to\-mor\-phisms, etc) for obtaining geometric integrators preserving symplecticity from discretization maps. We present here an equivalent definition of those integrators useful for the goals of the current paper.

In fact, if $L: TQ\rightarrow {\mathbb R}$ is a regular Lagrangian and $R_d: TQ\rightarrow Q\times Q$ is a discretization map, then using the Legendre map ${\mathcal F}L\colon TQ\rightarrow T^*Q$ and the cotangent lift of $R_d$ we can write the algorithm in Table 1 where $h\, dL(TQ)\subset T^*TQ$ is understood 
as the product by $h$ of the elements of $dL(TQ)$  with respect to the vector bundle structure $\zeta_{Q}: T^*TQ\rightarrow T^*Q$. Locally,
$
\zeta_Q(q, \dot{q}, p_q, p_{\dot{q}})=(q, p_{\dot{q}})
$ and
\begin{align*}\label{double-vector}
h\, dL(TQ)&=\{
(q, h\dot{q}, hp_q, p_{\dot{q}})\in T^*TQ\; \; | \;
(q, \dot{q}, p_q, p_{\dot{q}})\in dL(TQ)\}
\\
&=
\left\{ \left(q, h\dot{q}, h\frac{\partial L}{\partial q}, \frac{\partial L}{\partial \dot{q}}\right)\right\}\, .
\end{align*}

\begin{table}[h!]{{{\bf Table 1: Symplectic integrator} for a Lagrangian mechanical system 
}}
 \centering
  \begingroup
  \ttfamily\bfseries
  \renewcommand{\arraystretch}{1.2}
  \begin{tabular}{|l|}
   \hline
 input: $L: TQ\rightarrow {\mathbb R}$, $v_0\in TQ$ and $h>0$\\
     for \(k = 0\) to \(N-1\)\\
     \ \  define $p_k={\mathcal F}L(v_k)$\\
\ \ solve:  $p_{k+1}\in T^*Q$ such that $(-p_{k}, p_{k+1})\in  T^*R_d (h\, {\rm d}L(TQ))$\\
     end for\\
     solve: $v_N$ such that $p_N={\mathcal F}L(v_N)$\\
     output: \(v_{N}\)\\
     \hline
  \end{tabular}
\endgroup
\end{table}

Analogously, for a Hamiltonian function $H: T^*Q\rightarrow {\mathbb R}$, it is also needed to use the canonical antisymplectomorphism ${\mathcal I}_{ TQ}: T^*T^*Q\rightarrow T^*TQ$ between the symplectic manifolds $(T^*T^*Q, \omega_{T^*Q})$ and $(T^*TQ, \omega_{TQ})$ ~\cite{XuMac} and the cotangent lift $T^*R_d\colon T^*TQ\rightarrow T^*Q\times T^*Q$  
to write the algorithm in Table 2. Here, $h\, dH(T^*Q)$ is understood as a fiber multiplication with respect to the projection $\pi_{T^*Q}: T^*T^*Q\rightarrow T^*Q$, that is,
\begin{align*}
    h\, dH(T^*Q)&=
\left\{ \left(q, p, h\frac{\partial H}{\partial q}, h\frac{\partial H}{\partial p}\right)\right\}
\end{align*}
\begin{table}[h!]{{{\bf Table 2: Symplectic integrator} for a Hamiltonian mechanical system  
}}
  \centering
  \begingroup
  \ttfamily\bfseries
  \renewcommand{\arraystretch}{1.2}
  \begin{tabular}{|l|}
   \hline
 input: $H: T^*Q\rightarrow {\mathbb R}$, $p_0\in T^*Q$ and $h>0$\\
     for \(k = 0\) to \(N-1\)\\
\ \ solve:  $p_{k+1}\in T^*Q$ such that \\ \qquad $(-p_{k}, p_{k+1})\in  (T^*R_d\circ {\mathcal I}_{TQ}) (h \, {\rm d}H(T^*Q))$\\
     end for\\
     output: \(p_{N}\)\\
     \hline
  \end{tabular}
\endgroup
\end{table}

\subsection{Aim of the paper}

In the current work, using an important result on symplectic groupoids~\cite{coste} which relates Lagrangian bisections on a symplectic groupoid with Poisson transformations on the base space of the Lie groupoid, we generalize the previous methods to the case  of Lie algebroids and groupoids. We also elucidate the geometrical setting of the corresponding theories. 
As a consequence, the same discrete scheme is used to obtain at once geometric integrators for all the above mentioned cases and automatically obtaining methods preserving the corresponding Poisson structure defined on the dual of the Lie algebroid \cite{LeMaMa}. 

To be more precise, we consider the particular case when the Lie groupoid is a Lie group $G$. We will denote the left and right translations in $G$ by $L_g$ and $R_g$, respectively, for any $g\in G$. The geometric developments described in Section~\ref{Sec:CoLift} makes possible a general construction of Lie-Poisson integrators that can be algorithmically written as follows.  Let ${\mathfrak g}$ be the corresponding Lie algebra of $G$ and $\tau: {\mathfrak g}\rightarrow G$ be a retraction map \cite{Iserles2000,iserles,bourabee04}. Then, the methods for Lagrangian and Hamiltonian  systems given by a reduced Lagrangian function ${\bm l}: {\mathfrak g}\rightarrow {\mathbb R}$ and a reduced Hamiltonian function ${\bm h}: {\mathfrak g}^*\rightarrow {\mathbb R}$, respectively, are obtained in Table 3 and Table 4.
 
\begin{table}[h!]{{{\bf Table 3: Lie-Poisson integrator} for a Lagrangian system on ${\mathfrak g}$ 
}}
  \centering
  \begingroup
  \ttfamily\bfseries
  \renewcommand{\arraystretch}{1.2}
  \begin{tabular}{|l|}
  \hline
     input: ${\bm l}: {\mathfrak g}\rightarrow {\mathbb R}$, $\xi_0\in {\mathfrak g}$ and $h>0$\\
     define $\mu_0=\frac{\partial {\bm l}}{\partial \xi}(\xi_0)\in {\mathfrak g}^* $\\
     for \(k = 0\) to \(N-1\)\\
    \ \ solve:  $g_k\in G$ such that $(T_{g_k}R_{g^{-1}_{k}})^*(\mu_k)\in T^*\tau (h \, {\rm d} {\bm l}({\mathfrak g}))$\\
   \ \   define: $\mu_{k+1}=Ad^*_{g_k}\mu_k$\\
     end for\\
      solve: $\xi_N$ such that $\mu_N=\frac{\partial {\bm l}}{\partial \xi}(\xi_N)$\\
     output: \(\xi_{N}\in {\mathfrak g}\)\\
     \hline
  \end{tabular}
\endgroup
\end{table}

\begin{table}[h!]{{{\bf Table 4: Lie-Poisson integrator} for a Hamiltonian system on ${\mathfrak g}^*$ 
}}
  \centering
  \begingroup
  \ttfamily\bfseries
  \renewcommand{\arraystretch}{1.2}
  \begin{tabular}{|l|}
\hline
     input: ${\bm h}: {\mathfrak g}^*\rightarrow {\mathbb R}$, $\mu_0\in {\mathfrak g}^*$ and $h>0$\\
     for \(k = 0\) to \(N-1\)\\
  \ \   solve:  $g_k\in G$ such that $(T_{g_k}R_{g^{-1}_{k}})^*(\mu_k)\in (T^*\tau \circ {\mathcal I}_{\mathfrak g}) (h \, {\rm d} \, {\bm h}({\mathfrak g}^*))$\\
 \ \    define: $\mu_{k+1}=Ad^*_{g_k}\mu_k$\\
     end for\\
     output: \(\mu_{N}\in {\mathfrak g}^* \)\\
     \hline
  \end{tabular}
\endgroup
\end{table}
In Section~\ref{Sec:Integrator} we show that using these methods we obtain consistent integrators for the continuous dynamics that can be used as building blocks for constructing higher order integrators.

\subsection{Structure of the paper}

The paper is structured as follows. In Section~\ref{section1}, we review the notion of a discretization map on the pair groupoid. Section~\ref{Sec:General} contains the notion of a generalized discretization map on the more general setting of Lie groupoids. In Section~\ref{Sec:reduction}, discretization maps are obtained from a reduction process because of the existence of symmetries in a mechanical system. Sections~\ref{Sec:Lift} and~\ref{Sec:CoLift} describe the tangent and cotangent lifts of generalized discretization maps on Lie groupoids. With all the  constructions already described and the theory of symplectic groupoids, Poisson integrators for reduced systems are obtained in Section~\ref{Sec:Integrator}. Such methods are applied to explicit examples such as a rigid body and a heavy top. Finally, Section~\ref{Sec:conclusion} exposes the conclusions and some future research lines with some possible applications in optimization of neural networks with symmetry \cite{10.5555/3540261.3542225,Zhao2022SymmetriesFM} and in obtaining geometric integrators for second order differential equations on Lie algebroids. The well-known results and constructions for Lie groupoids and symplectic groupoids have been summarized in the appendices. These results are needed along the paper.

\section{Retraction maps and discretization maps}\label{section1}
In~\cite{AbMaSeBookRetraction,Shub} a retraction map on a manifold $Q$ is a smooth map 
$R\colon U\subseteq TQ \rightarrow Q$ where $U$ is an open subset containing the zero section  of the tangent bundle 
such that the restriction map $R_q=R_{|T_qQ}\colon T_qQ\rightarrow Q$ satisfies
\begin{enumerate}
	\item $R_q(0_q)=q$ for all $q\in Q$,

	\item 
	${\rm D}R_q(0_q)=T_{0_q}R_q={\rm Id}_{T_qQ}$ with the identification $T_{0_q}T_qQ\simeq T_qQ$.
\end{enumerate}

\begin{example}
If  $(Q, g)$ is a Riemannian manifold, then the exponential map $\hbox{exp}^g: U\subset TQ\rightarrow Q$ is a typical example of a retraction map: 
$
\hbox{exp}^g_q(v_q)=\gamma_{v_q}(1), 
$
where $\gamma_{v_q}$ is the unique  Riemannian geodesic~\cite{doCarmo} satisfying $\gamma_{v_q}(0)=q$ and $\gamma'_{v_q}(0)=v_q$. 
\end{example}

In \cite{21MBLDMdD}  retraction maps have been used to define discretization maps $R_d\colon U \subset TQ  \rightarrow Q\times Q$, where  $U$ is an open neighbourhood of the zero section of $TQ$, 
\begin{eqnarray*}
	R_d\colon U \subset TQ & \longrightarrow & Q\times Q\\
	v_q & \longmapsto & (R^1(v_q),R^2(v_q))\, .
\end{eqnarray*}
Discretization maps satisfy the following properties:
\begin{enumerate}
	\item $R_d(0_q)=(q,q)$, for all $q\in Q$.
	\item $T_{0_q}R^2_q-T_{0_q}R^1_q={\rm Id}_{T_qQ} \colon T_{0_q}T_qQ\simeq T_qQ \rightarrow T_qQ$ is equal to the identity map on $T_qQ$ for any $q$ in $Q$.
\end{enumerate}
As a consequence, the discretization map $R_d$ is a local diffeomorphism.

\begin{example}
Examples of discretization maps on Euclidean vector spaces are:
\begin{itemize}
	\item Explicit Euler method:  $R_d(q,v)=(q,q+v).$
	\item Midpoint rule:  $R_d(q,v)=\left( q-\dfrac{v}{2}, q+\dfrac{v}{2}\right).$
		\item $\theta$-methods with $\theta\in [0,1]$:  \hspace{3mm} $R_d(q,v)=\left( q-\theta \, v, q+ (1-\theta)\, v\right).$
\end{itemize}
\end{example}

An interesting observation deals with the underlying geometry in the retraction and discretization maps. Note that $Q\times Q$ has a Lie groupoid structure and $TQ$ is the corresponding Lie algebroid (see Appendix~\ref{App:lieGroupoid}). It will be shown in Section~\ref{Sec:General} that the Lie groupoid is the geometric structure that makes possible to extend the notion of discretization maps so that geometric integrators can be defined for a wider family of mechanical systems. Many of the properties satisfied by the above mentioned discretization maps have to deal with the algebraic structure associated with the Lie groupoid (set of identities, inversion map...).

\section{Generalized discretization maps on Lie grou\-poids}\label{Sec:General}
In a previous paper~\cite{24IFAC}, we have analyzed the relation between the symmetry of the numerical method and the preservation of constants of motion associated to the original continuous momentum map. The Lie group invariance leads to the reduction of the system. That is why it is worthwhile to define discretization maps for reduced systems following the framework introduced in \cite{weinstein96} once the notion of retraction map given in \cite{AbMaSeBookRetraction} is extended to the setting of Lie algebroids and Lie groupoids. 

Let us first motivate the above-mentioned extension. A discretization map is a smooth map $R_d\colon TQ \rightarrow Q\times Q$ that satisfies two conditions recalled in the previous section. On the other hand, it is well-known that $Q\times Q$ is a Lie groupoid with Lie algebroid the tangent bundle of $Q$, that is, $TQ$. In fact, the map \begin{eqnarray*}
    TQ& \longrightarrow & A(Q\times Q) \\ v_q & \longmapsto & (0_q,v_q),
\end{eqnarray*}
is a vector bundle isomorphism between $TQ$ and the Lie algebroid $A(Q\times Q)$ (see Appendix~\ref{App:ExamplePairGroup}). Some notation is introduced to state the following result. The zero section of $TQ$ is denoted by $0\colon Q\rightarrow TQ$, $\epsilon\colon Q\rightarrow Q\times Q$ is the identity section of $Q\times Q$, $\epsilon=\Delta\colon Q\rightarrow Q\times Q$ is the diagonal map given by $\Delta(q)=(q,q)$, $\alpha\colon Q\times Q \rightarrow Q$ is the source map given by the canonical projection onto the first factor, that is, $\alpha={\rm pr}_1\colon Q\times Q \rightarrow Q$ (for more details see Appendix~\ref{App:ExamplePairGroup}).

After introducing the structural maps of the Lie algebroid $TQ$ and the Lie groupoid $Q\times Q$, the properties to be satisfied by the discretization maps can be rewritten in terms of those structural maps. 

\begin{proposition}\label{Prop:previous}
    Let $R_d\colon TQ \rightarrow Q\times Q$ be a smooth map. It is a discretization map if and only if:
    \begin{enumerate}
        \item For every $q\in Q$, $R_d(0_q)=\epsilon (q,q)=(q,q)$, that is, the following diagram is commutative:
        \begin{equation*}
\xymatrix{	TQ \ar[rr]^{R_d} && Q\times Q \\ & Q \ar[ul]^{0} \ar[ur]_{\epsilon}& }
\end{equation*}  
\item If $R_q=R_d|_{T_qQ}$ for every $q\in Q$, then 
$$T_{0_q}R_q\left((v_q)^{\mathtt{v}}_{0_q} \right)=(0_q,v_q)+\left[T_q\epsilon \circ T_{\epsilon(q)}\alpha \circ T_{0_q}R_q\right]((v_q)^{\mathtt{v}}_{0_q}), $$
for all $v_q\in T_qQ$, where $(v_q)^{\mathtt{v}}_{0_q}$ denotes the vertical lift given by $\left.\dfrac{\rm d}{{\rm d}t}\right|_{t=0}(tv_q)$.
    \end{enumerate}
\end{proposition}
\begin{proof}
    The first condition corresponds with the first condition in the definition of a discretization map recalled in Section~\ref{section1}.

    Let us prove now the second condition. If $R_d(v_q)=(R^1(v_q),R^2(v_q))$, we must prove that $T_{0_q}R^2_q-T_{0_q}R^1_q\colon T_{0_q}(T_qQ)\simeq T_qQ \rightarrow T_qQ$ is the identity map on $T_qQ$, for all $q\in Q$, if and only if the second condition in the statement holds. 

    If $v_q\in T_qQ$, a straightforward computation shows that 
\begin{equation*}T_{0_q}R_q\left((v_q)^{\mathtt{v}}_{0_q} \right)=\left( T_{0_q}R^1_q(v_q), \;T_{0_q}R^2_q(v_q) \right),\end{equation*}
and \begin{equation*}(0_q,v_q)+\left[T_q\epsilon \circ T_{\epsilon(q)}\alpha \circ T_{0_q}R_q\right]((v_q)^{\mathtt{v}}_{0_q})=\left( T_{0_q}R^1_q(v_q), \;v_q+T_{0_q}R^1_q(v_q) \right) .\end{equation*}
    This concludes the proof.
    
\end{proof}

Now, let $G$ be an arbitrary Lie groupoid over $Q$ with the source and target maps $\alpha,\beta\colon G\rightarrow Q$, respectively, and the identity section $\epsilon\colon Q \rightarrow G$ such that the dimensions of $Q$ and $G$ are $n$, $n+r$, respectively. Another structural maps associated with a groupoid structure are the multiplication map $m\colon G_2\rightarrow G$, where $G_2$ is the set of composable maps, and the inversion map $i\colon G\rightarrow G$.  A Lie groupoid is denoted by $G\rightrightarrows Q$. The corresponding Lie algebroid of $G$ is  a real vector bundle $\tau\colon AG\rightarrow Q$ whose fibers are given by the vertical fiber of the source map, that is, $A_qG=V_{\epsilon(q)}\alpha$ for every $q\in Q$. For more
 	information about this concept, we refer the reader to Appendix~\ref{App:lieGroupoid} and  the monographs \cite{2021CrainicBook} and \cite{Mackenzie}.

Proposition~\ref{Prop:previous} leads to the following generalized definition.

\begin{definition} \label{Def:newdiscretization}
    A smooth map ${\mathcal R}\colon AG \rightarrow G$ is a generalized discretization map on the Lie groupoid $G$ if it satisfies the following properties:
    \begin{enumerate}
        \item ${\mathcal R}\circ 0_{AG}= \epsilon$, where $0_{AG}\colon Q \rightarrow AG$ is the zero section of $AG$, that is, the following diagram is commutative: 
\begin{equation*}
\xymatrix{	AG \ar[rr]^{\mathcal R} && G \\ & Q \ar[ul]^{0_{AG}} \ar[ur]_{\epsilon}& }
\end{equation*}  
\item If ${\mathcal R}_q= \left.{\mathcal R}\right|_{A_qG}$ for every $q\in Q$, then as $\alpha(\epsilon(q))=q$,
$$\left(T_{0_{AG}(q)}{\mathcal R}_q\right)\left(\left(v_q\right)^{\mathtt{v}}_{0_{AG}(q)}\right)=v_q+\left[ (T_q\epsilon)\circ (T_{\epsilon(q)}\alpha) \circ (T_{0_{AG}(q)}{\mathcal R}_q)\right]\left(\left(v_q\right)^{\mathtt{v}}_{0_{AG}(q)}\right)$$
for all $v_q\in A_qG$, where $\left(v_q\right)^{\mathtt{v}}_{0_{AG}(q)}$ denotes the vertical lift given by $\left.\dfrac{\rm d}{{\rm d}t}\right|_{t=0}(t v_q)$.
    \end{enumerate}
\end{definition}
Note that using {\it 1.} in Definition~\ref{Def:newdiscretization}, $T_{0_{AG}(q)}{\mathcal R}_q\colon T_{0_{AG}(q)}(A_qG) \simeq A_qG\rightarrow T_{\epsilon(q)}G$ and \begin{equation}\label{eq:splittingTG} T_{\epsilon(q)}G\simeq A_qG\oplus (T_q\epsilon)(T_qQ) \quad \forall \; q\in Q.\end{equation}

In fact, the linear map \begin{eqnarray*}
T_{\epsilon(q)}G&\longrightarrow &A_qG \oplus (T_q\epsilon) (T_qQ) 
\\ v_{\epsilon(q)}& \longrightarrow &\left(v_{\epsilon(q)}- (T_q\epsilon)(T_{\epsilon(q)}\alpha)(v_{\epsilon(q)}),(T_q\epsilon)(T_{\epsilon(q)}\alpha)(v_{\epsilon(q)})\right)
\end{eqnarray*}
is the isomorphism that guarantees Equation~\eqref{eq:splittingTG}. 

Proposition~\ref{Prop:previous} guarantees that the notion in Definition~\ref{Def:newdiscretization} includes the concept of discretization maps in~\cite{21MBLDMdD} and generalizes that notion to Lie groupoids. As a consequence, reduced systems, optimal control problems with symmetries and many other problems can be studied using the geometric tools provided in this paper.

Next, we introduce the adjoint of a generalized discretization on a Lie groupoid $G$ using the inversion map $i: G\rightarrow G$ in $G$.

\begin{proposition}
If ${\mathcal R}: AG\rightarrow G$ is a generalized discretization map, then 
${\mathcal R}^*: AG\rightarrow G$ defined by 
${\mathcal R}^*(v_q)=(i\circ {\mathcal R})(-v_q)$ is also a generalized discretization map and it is called the adjoint of ${\mathcal R}$.
\end{proposition}
\begin{proof}
We have that:
\begin{enumerate}
    \item Obviously ${\mathcal R}^*\circ 0_{AG}=\epsilon$ since $(i\circ {\mathcal R})(-0_{AG}(q))=(i\circ \epsilon)(q)=\epsilon(q)$.
\item Moreover, if $v_q\in A_qG$, we have that
\begin{align*}
& \left(T_{0_{AG}(q)}{\mathcal R}^*_q\right)\left(\left(v_q\right)^{\mathtt{v}}_{0_{AG}(q)}\right)=
\left(T_{0_{AG}(q)}(i\circ {\mathcal R}_q\right)\left(\left(-v_q\right)^{\mathtt{v}}_{0_{AG}(q)}\right)\\
&=T_{\epsilon(q)}i\left[
-v_q-\left[ (T_q\epsilon)\circ (T_{\epsilon(q)}\alpha) \circ (T_{0_{AG}(q)}{\mathcal R}_q)\right]\left(\left(v_q\right)^{\mathtt{v}}_{0_{AG}(q)}\right)\right]
\\
&=v_q-T_{\epsilon(q)}(\epsilon\circ \beta)(v_q)-\left[(T_q\epsilon)\circ(T_{\epsilon(q)}\alpha) \circ (T_{0_{AG}(q)}{\mathcal R}_q)\right]\left(\left(v_q\right)^{\mathtt{v}}_{0_{AG}(q)}\right)
\\
&=v_q-\left[(T_q\epsilon)\circ(T_{\epsilon(q)}\beta)\right]\left(v_q+\left[ (T_q\epsilon)\circ (T_{\epsilon(q)}\alpha) \circ (T_{0_{AG}(q)}{\mathcal R}_q)\right]\left(\left(v_q\right)^{\mathtt{v}}_{0_{AG}(q)}\right)\right)\\
&=v_q-\left[ (T_q\epsilon)\circ (T_{\epsilon(q)}\beta)\circ  (T_{0_{AG}(q)}{\mathcal R}_q)\right]\left(\left(v_q\right)^{\mathtt{v}}_{0_{AG}(q)}\right)\\
&=v_q+\left[ (T_q\epsilon)\circ (T_{\epsilon(q)}\alpha) \circ (T_{0_{AG}(q)}{\mathcal R}^*_q)\right]\left(\left(v_q\right)^{\mathtt{v}}_{0_{AG}(q)}\right)
\end{align*}
where we use that 
\[
T_{\epsilon(q)}i(v_q)=-v_q+T_{\epsilon(q)}(\epsilon\circ \beta)(v_q)
\]
and the following relations: $i\circ \epsilon=\epsilon$, $\epsilon\circ \beta\circ \epsilon=\epsilon$ and $\beta\circ i=\alpha$.
\end{enumerate}
\end{proof}

Finally, we introduce the following definition. 
\begin{definition}
   We say that ${\mathcal R}: AG\rightarrow G$ is a {\bf symmetric generalized discretization map} if 
   ${\mathcal R}={\mathcal R}^*$.
\end{definition}

\subsection{Local expression of a generalized discretization map}\label{Sec:Local}
In order to work with generalized discretization maps on Lie groupoids and derive geometric integrators, the local characterization is useful.  Details on local description for discrete mechanics on Lie groupoids can be found in~\cite{2015JCDavidLocalDiscrete}.

Given a point $q_0\in Q$ take local coordinates $(q^i,u^\gamma)$ on $G$ for $i=1,\dots,n=\dim Q$ and $\gamma=1,\dots, r=\dim G-\dim Q$, defined in a neighbourhood ${\mathcal U}$ of $\epsilon(q_0)$, adapted to the source map $\alpha$. Therefore, $\alpha (q^i,u^\gamma)=(q^i)$. Additionally we assume that the identities are elements with coordinates $(q^i, u^{\gamma}=0)$.

 A local basis of the sections of $\tau\colon AG\rightarrow Q$ is given by: $\left\{ e_\gamma=\dfrac{\partial}{\partial u^\gamma} \circ \epsilon \right\}$ where we use the identification $A_{q}G\equiv V_{\epsilon(q)}\alpha=\ker (T_{\epsilon(q)}\alpha)$.
Therefore, we induce local coordinates on the Lie algebroid $AG$ as $(q^i,y^\gamma)$.

Thus, a local basis using  the splitting in Equation~\eqref{eq:splittingTG} is:
\begin{eqnarray*}
    T_{\epsilon(q)}G&\simeq & A_qG \oplus (T_q\epsilon)(T_qQ)\\
    \left\langle \left.\dfrac{\partial}{\partial q^i}\right|_{\epsilon(q)}, \left.\dfrac{\partial}{\partial u^\gamma}\right|_{\epsilon(q)}\right\rangle&\simeq&   \left\langle  \left.\dfrac{\partial}{\partial y^\gamma}\right|_{\epsilon(q)} \right\rangle \oplus  \left\langle (T_q\epsilon) \left( \left.\dfrac{\partial}{\partial q^i}\right|_{q}\right)\right\rangle \, .
\end{eqnarray*}
Locally, a smooth map
\begin{equation*}
\begin{array}{lrl}
    {\mathcal R}\colon & AG \longrightarrow & G\\
    & (q^i,y^\gamma)  \longmapsto & \left({\mathcal R}^i(q,y),{\mathcal R}^\gamma(q,y)\right)\, ,
    \end{array}
\end{equation*}
is a generalized discretization map on a Lie groupoid if
\begin{enumerate}
    \item ${\mathcal R}^i(q,0)=q^i$ and ${\mathcal R}^\gamma(q,0)=0$ for all $i=1,\dots, n$ and $\gamma=1,\dots, r$.
    \item $\left( \left.\dfrac{\partial {\mathcal R}^\gamma}{\partial y^{\gamma'}}\right|_{\epsilon(q)=(q,0)}\right)={\rm Id}_{r\times r}$, where ${\rm Id}_{r\times r}$ is the identity matrix of size $ r\times r$.
\end{enumerate}

This local description of Definition~\ref{Def:newdiscretization} makes easier to prove the following result. 

\begin{theorem}\label{thm:diffeo} If ${\mathcal R}\colon AG\rightarrow G$ is a generalized discretization map on the Lie groupoid $G$, then ${\mathcal R}$ is a diffeomorphism from an open neighbourhood of the zero section of $AG$ to an open neighbourhood of the identity section of $G$.
\end{theorem}

\begin{proof}
The above-mentioned local coordinates are used to complete the proof. The map ${\mathcal R}\colon AG\rightarrow G$ is a local diffeomorphism if the Jacobian matrix at $\epsilon(q)=(q,0)$ is non-singular. Locally,
$$ J{\mathcal R} (q,0)=\begin{pmatrix}
\dfrac{\partial {\mathcal R}^i}{\partial q^j} & \dfrac{\partial {\mathcal R}^i}{\partial y^\gamma} \\[1em]
\dfrac{\partial {\mathcal R}^\gamma}{\partial q^j} & \dfrac{\partial {\mathcal R}^\gamma}{\partial y^{\gamma'}}
\end{pmatrix}_{(q,0)}=\begin{pmatrix} {\rm Id}_{n\times n} & \dfrac{\partial {\mathcal R}^i}{\partial y^\gamma} \\   0 &  {\rm Id}_{r\times r}\end{pmatrix}_{(q,0)}$$
because of the local expression of the two properties in Definition~\ref{Def:newdiscretization}. 
\end{proof}

It is clear that the discretization maps on the pair groupoid are precisely the discretization maps introduced in~\cite{21MBLDMdD}. Now, some examples of generalized discretization maps on other Lie groupoids are described because they play a key role to obtain numerical integrators for different mechanical systems as shown in Section~\ref{Sec:Integrator}.

\begin{example} \label{Ex:LieAlgebra}
    Let $G$ be a Lie group. The Lie groupoid $G\rightrightarrows \{e\}$ has the Lie algebra $\mathfrak{g}=T_eG$ as Lie algebroid where $e$ is the identity element in the Lie group (see Appendix~\ref{App:ExampleLieGroup}). A generalized discretization map ${\mathcal R}\colon \mathfrak{g} \rightarrow G$ on $G$ is a smooth map that satisfies 
    \begin{enumerate}
        \item ${\mathcal R}(0_e)=e$,
        \item $T_{0_e}{\mathcal R}\colon T_{0_e} {\mathfrak{g}}\simeq {\mathfrak g}\rightarrow {\mathfrak g}=T_eG$ is the identity map.
    \end{enumerate}
    That construction is useful, for instance, for mechanical problems that model the rigid body equations (see \cite{bourabee04} and Example \ref{section7.4}).
\end{example}

\begin{example} Let $\pi\colon Q\rightarrow M$ be a surjective submersion. The set
$$G=Q\times_\pi Q=\{(q,q')\in Q\times Q\; | \; \pi(q)=\pi(q')\}$$
is a Lie subgroupoid $G\rightrightarrows Q$ of the pair groupoid with Lie algebroid $\tau\colon AG=V\pi\rightarrow Q$, where $V\pi$ is the vertical subbundle associated with $\pi$, that is, $V\pi=\{X\in TQ\; |\; T\pi(X)=0\}$. A smooth map ${\mathcal R}\colon V\pi \rightarrow Q\times_\pi Q$, ${\mathcal R}(X_q)=({\mathcal R}^1(X_q),{\mathcal R}^2(X_q))$ for $X_q\in V_q\pi$,  is a generalized discretization on $G$ if it satisfies
\begin{enumerate}
    \item ${\mathcal R}(0_q)=(q,q)$,
    \item $T_{0_q}{\mathcal R}^2_q-T_{0_q}{\mathcal R}^1_q\colon V_q\pi \rightarrow T_qQ$ is the canonical inclusion over $V_q\pi$ in $T_qQ$.
\end{enumerate}
    This is related to the geometric integrators for holonomic mechanics described in the preprint~\cite{2023BarMar} since $V\pi$ defines an integrable foliation and each leaf of this foliation can be identified with a submanifold of $Q$.
\end{example}

\begin{example}\label{example-action} Let $H$ be a Lie group and 
$\cdot: Q \times H \rightarrow Q$, $(q, {\mathbf h})\mapsto q {\mathbf h}$, be a right action of $H$ on $Q$. Consider the Lie groupoid 
$G=Q \rtimes H \rightrightarrows Q$ called the right action groupoid,  with source map $\alpha(q,{\mathbf h})=q$, target map $\beta(q,{\mathbf h})=q\, {\mathbf h}$, identity section $\epsilon(q)=(q,e)$, inversion map $i(q,{\mathbf h})=(q\, {\mathbf h},  {\mathbf h}^{-1})$ and product map
\[
m((q,{\mathbf h}), (q{\mathbf h}, {\mathbf h}'))=(q,{\mathbf h}{\mathbf h}')\, .
\]
The  Lie algebroid of the right action groupoid is just the right action algebroid   $\tau\colon AG=Q\rtimes \mathfrak{h}\rightarrow Q$, where $\mathfrak{h}$ is the Lie algebra of $H$ (see Appendix~\ref{App:action}).  A generalized discretization map ${\mathcal R}\colon Q \rtimes \mathfrak{h}\rightarrow Q \rtimes H$, ${\mathcal R}(q, \xi)=({\mathcal R}^1(q,\xi),{\mathcal R}^2(q,\xi))$ is a smooth map such that 
\begin{enumerate}
    \item ${\mathcal R}(q,0_{\mathfrak{h}})=(q,e)$,
    \item $T_{0_{\mathfrak{h}}}{\mathcal R}_q\colon \mathfrak{h}\rightarrow T_qQ\times \mathfrak{h}$ is given by $\xi\mapsto \left( T_{0_h}{\mathcal R}^1_q(\xi),\xi \right)$.
\end{enumerate}

     That construction is useful for mechanical problems that model, for instance, the heavy top (see Example \ref{example:heavy}).
\end{example}
\begin{example}
The Riemmanian exponential map is the most natural retraction (or discretization map) to use on a Riemannian manifold \cite{AbMaSeBookRetraction}.  This idea can be generalized to the case when we have  a non-degenerate bundle metric  $g: AG\times_Q AG\rightarrow {\mathbb R}$  on a Lie algebroid.
This bundle metric defines a homogeneous quadratic SODE $\Gamma_g$ on $AG$  and its associated exponential map 
\[
\begin{array}{rrcl}
\hbox{exp}_{\Gamma_g}:& U\subseteq AG&\longrightarrow&  G
\end{array}
\]
where  $U$  is an  open subset of $AG$ over  the restriction of the zero section to  $\tau(U)\subseteq Q$ (see details in \cite{MMM3-1}). The map $\hbox{exp}_{\Gamma_g}$ is another example of  a generalized discretization map and, in general, it  is also a generalized discretization map, the exponential map associated to any homogeneous quadratic SODE (see Section 4 in \cite{MMM3-1}). 

\end{example}
\section{Reduction of generalized discretization maps}\label{Sec:reduction}

As stated in Noether's theorem, the existence of symmetries in mechanical systems  makes possible to reduce the original system to a system with less degrees of freedom, see for instance~\cite{AM87}. We are interested in discretizing the original system so that the conserved quantities are preserved after reducing the generalized discretization maps described in Section~\ref{Sec:General}.

First, following Appendix~\ref{App:lieGroupoid},  a Lie groupoid isomorphism $(\Phi,\phi)$ between the Lie groupoids $G\rightrightarrows Q$ and $G'\rightrightarrows Q'$ defines the corresponding Lie algebroid isomorphism $(A\Phi,\phi)$ such that $ A\Phi(v_{q})=T_q\Phi (v_{q})\in AG'$ for all 
$v_{q}\in A_qG$. Having in mind the following commutative diagrams,
 \begin{equation*}
\xymatrix{
	G\ar[rr]^{\Phi} \ar[d]<2pt>^{\beta} \ar[d]<-2pt>_{\alpha} && G' \ar[d]<2pt>^{\beta'} \ar[d]<-2pt>_{\alpha'} && AG\ar[rr]^{A\Phi} \ar[d]^{\tau} && AG' \ar[d]^{\tau'}  \\ 
	Q  \ar[rr]^{\phi} && Q' && Q  \ar[rr]^{\phi} && Q' }
\end{equation*} 
it is possible to use a discretization map ${\mathcal R}\colon AG\rightarrow G$ on the Lie groupoid $G$ to define a discretization map ${\mathcal R}'\colon AG'\rightarrow G'$ on the Lie groupoid $G'$ and vice-versa. 

\begin{theorem} \label{th:GtoG'} Let $(\Phi,\phi)$ and $(A\Phi,\phi)$ be a Lie groupoid isomorphism and the corresponding  Lie algebroid isomorphism, respectively. If ${\mathcal R}\colon AG\rightarrow G$ is a generalized discretization map on the Lie groupoid $G$, then ${\mathcal R}'=\Phi\circ {\mathcal R}\circ A\Phi^{-1}\colon AG'\rightarrow G'$ is a generalized discretization map on the Lie groupoid $G'$. 
\end{theorem}
\begin{proof}
    We have the following diagrams
 \begin{equation*}
\xymatrix{
Q \ar[rr]<-4pt>_{0_{AG}}  \ar[d]^{\phi} && AG  \ar[ll]<-4pt>_{\tau} \ar[rr]^{\mathcal R} && G \ar[d]^{\Phi}  \ar[rr]<4pt>_{\beta} \ar[rr]<8pt>^{\alpha}  && Q \ar[d]^{\phi} \ar[ll]<6pt>^{\epsilon} \\  Q'\ar[rr]<-4pt>_{0_{AG'}}  && AG'   \ar[u]^{A\Phi^{-1}} \ar[ll]<-4pt>_{\tau'}  \ar[rr]^{{\mathcal R}'} && G'  \ar[rr]<6pt>^{\alpha'}  \ar[rr]<2pt>_{\beta'} && Q' \ar[ll]<9pt>^{\epsilon'} }
\end{equation*} 
    Then, using the definition of a Lie groupoid isomorphism, the following facts are obtained:
    \begin{align}
        &A\Phi^{-1}\circ 0_{AG'}=0_{AG}\circ \phi^{-1},\label{equation-1}\\
       &\Phi\circ \epsilon=\epsilon'\circ \phi,\label{equation-2}\\
       & \phi\circ \alpha=\alpha'\circ \Phi,\label{equation-3}\\ 
    &{\mathcal R}'_{\phi(q)}\circ (A\Phi)_{A_qG}=\Phi\circ {\mathcal R}_q \quad \mbox{ for every  } q\in Q, \label{equation-4}
    \end{align}
    because of the definition of ${\mathcal R}'$. In particular, the previous diagrams are commutative.
    Now, if $v_q\in A_qG$ then from (\ref{equation-4}) we deduce that
    \begin{align*}
    T_{0_{AG'}(\phi(q))}{\mathcal R}'((A\Phi)(v_q))^{\rm v}_{0_{AG'}(\phi(q))}&=
    \frac{d}{dt}\Big|_{t=0}\left((\Phi\circ {\mathcal R})(tv_q)\right)\\
    &=(T_q\Phi)(T_{0_{AG}(q)}{\mathcal R})\left((v_q)^{\rm v}_{0_{AG}(q)}\right)\, .
    \end{align*}
Then, using that ${\mathcal R}$ is a generalized discretization map (see Definition \ref{Def:newdiscretization}) it follows that 
 \begin{align*}
    T_{0_{AG'}(\phi(q))}{\mathcal R}'((A\Phi)(v_q))^{\rm v}_{0_{AG'}(\phi(q))}&=A\Phi(v_q)\\&+(T_{0_{AG}(q)}(\Phi\circ \epsilon\circ \alpha\circ {\mathcal R}_q))(v_q)^{\rm v}_{0_{AG}(q)}\, .
\end{align*}
Together with~\eqref{equation-1},~\eqref{equation-2} and~\eqref{equation-3}, we have 
 \begin{align*}
    T_{0_{AG'}(\phi(q))}{\mathcal R}'((A\Phi)(v_q))^{\rm v}_{0_{AG'}(\phi(q))}&=A\Phi(v_q)\\&+(T_{0_{AG}(q)}( \epsilon'\circ \alpha'\circ {\mathcal R}'_{\phi(q)}\circ (A\Phi)_{|A_qG})(v_q)^{\rm v}_{0_{AG}(q)}\\
    &=A\Phi(v_q)\\&+(T_{0_{AG'}(\phi(q))}( \epsilon'\circ \alpha'\circ {\mathcal R}'_{\phi(q)})((A\Phi)(v_q)^{\rm v}_{0_{AG}(q)})\, .
\end{align*}
As $A\Phi: AG\rightarrow AG'$ is a Lie algebroid isomorphism, the above equality proves that ${\mathcal R}'$ is a generalized discretization map. 
\end{proof}

The generalized discretization map ${\mathcal R}'$ in Theorem~\ref{th:GtoG'} is said to be isomorphic to the generalized discretization map ${\mathcal R}$.

Now, a free and proper action of a Lie group $H$ on the Lie groupoid $G\rightrightarrows Q$, $$\Phi\colon H\times G \rightarrow G,$$ by Lie groupoid isomorphisms may be used to develop reduction as follows. For every ${\mathbf h}\in H$, the pair $(\Phi_{\mathbf h},\phi_{\mathbf h})$ is a Lie groupoid isomorphism where $\Phi_{\mathbf h}(g)=\Phi({\mathbf h},g)$  for $g\in G$. Then, the following diagram is commutative:
 \begin{equation*}
\xymatrix{
	G\ar[rr]^{\Phi_{\mathbf h}} \ar[d]<2pt>^{\beta} \ar[d]<-2pt>_{\alpha} && G \ar[d]<2pt>^{\beta} \ar[d]<-2pt>_{\alpha}  \\ 
	Q  \ar[rr]^{\phi_{\mathbf h}} &&  Q }
\end{equation*} 

The action $(\Phi,\phi)$ of $H$ on $G\rightrightarrows Q$ induces an action $(A\Phi,\phi)$ of $H$ on the Lie algebroid $AG\rightarrow Q$ by Lie algebroid isomorphisms. In fact, for every ${\mathbf h}\in H$  the following diagram is commutative:
 \begin{equation*}
\xymatrix{
	AG\ar[rr]^{A\Phi_{\mathbf h}} \ar[d]^{\tau} && AG \ar[d]^{\tau}  \\ 
	Q  \ar[rr]^{\phi_{\mathbf h}} &&  Q }
\end{equation*} 
As $\phi$ is a free and proper action of $H$ on $Q$, the pair $(\Phi,\phi)$ is a free and proper action of $H$ on $G$ and $(A\Phi,\phi)$ is also a free and proper action of $H$ on $AG$ by Lie algebroid isomorphisms. Thus, the  quotient manifolds $G/H$ and $Q/H$ are well-defined. The objects associated with the quotient manifolds will be denoted by $\widehat{\cdot}$.

\begin{proposition} Let $\phi$ be a free and proper action of $H$ on $Q$. The quotient manifold $\widehat{G}=G/H$ is a Lie groupoid over $\widehat{Q}=Q/H$ with Lie algebroid $\widehat{AG}=AG/H\stackrel{\widehat{\tau}}{\rightarrow} \widehat{Q}$ in such a way that the couple $(\pi_G,\pi_Q)$ is a Lie groupoid epimorphism and $(\pi_{AG},\pi_Q)$ is the associated Lie algebroid epimorphism, where $\pi_G\colon G\rightarrow \widehat{G}$, $\pi_{AG}\colon AG \rightarrow \widehat{AG}$ and $\pi_Q\colon Q\rightarrow \widehat{Q}$ are the canonical projections. 
\end{proposition}

\begin{proof}
    From the definition of a Lie groupoid morphism (see Appendix~\ref{App:lieGroupoid}), we deduce that the structural maps $\alpha$, $\beta$, $\epsilon$ and $i$ of the Lie groupoid $G\rightrightarrows Q$ are $H$-equivariant. Then, the structural maps of $\widehat{G}=G/H\rightrightarrows \widehat{Q}=Q/H$ are the smooth maps
    $$\widehat{\alpha},\; \widehat{\beta}\colon \widehat{G}\rightarrow \widehat{Q}, \quad \widehat{\epsilon}\colon \widehat{Q}\rightarrow \widehat{G},\quad \widehat{i}\colon \widehat{G}\rightarrow \widehat{G},$$
    characterized by the following conditions:
    \begin{equation}\label{eq:structuralReduced} 
      \begin{array}{lll}  \widehat{\alpha}\circ \pi_G=\pi_Q\circ \alpha, && \widehat{\beta}\circ \pi_G=\pi_Q\circ \beta,\\
          \widehat{\epsilon}\circ \pi_Q=\pi_G\circ \epsilon, && \widehat{i}\circ \pi_G=\pi_G\circ i\, .\end{array}
    \end{equation}
    In addition, let $(\widehat{G}\times \widehat{G})_2$ be the set of composable pairs, that is, all the points $(\widehat{g}_1,\widehat{g}_2)\in \widehat{G}\times \widehat{G}$ such that $\widehat{\beta}(\widehat{g}_1)=\widehat{\alpha}(\widehat{g}_2)$. There exists a smooth map $\widehat{m}\colon (\widehat{G}\times \widehat{G})_2\rightarrow \widehat{G}$, the partial multiplication of $\widehat{G}$, such that the following diagram
\begin{equation*}
\xymatrix{
	 (G\times G)_2\ar[rr]^{m}  \ar[d]_{\pi_G\times {\pi_G}}   && G \ar[d]^{\pi_G}  \\ 
	 (\widehat{G}\times \widehat{G})_2 \ar[rr]^{\widehat{m}} &&  \widehat{G} }
\end{equation*} 
    is commutative 
    because the partial multiplication $m\colon (G\times G)_2\rightarrow G$ is $H$-equivariant.

    From the definition of the structural maps $\widehat{\alpha}$, $\widehat{\beta}$ and $\widehat{m}$
    of the Lie groupoid $\widehat{G}\rightrightarrows \widehat{Q}$, it is obtained that the canonical projection $\pi_G\colon G\rightarrow \widehat{G}$ is a Lie groupoid epimorphism over $\pi_Q\colon Q\rightarrow \widehat{Q}$.

    As for every $\mathbf{h}\in H$, $A\Phi_{\mathbf{h}}\colon AG \rightarrow AG$ is a Lie algebroid isomorphism over $\phi_{\mathbf{h}}\colon Q\rightarrow Q$ and proceeding as in the Lie groupoid cases, we obtain that there exists a Lie algebroid structure on the quotient vector bundle $\widehat{AG}=AG/H\rightarrow \widehat{Q}$ such that the canonical projection $\pi_{AG}\colon AG \rightarrow \widehat{AG}$ is a Lie algebroid epimorphism over $\pi_Q\colon Q\rightarrow \widehat{Q}$.

    As $\tau\colon AG \rightarrow Q$ is the Lie algebroid of the Lie groupoid $G\rightrightarrows Q$, it is satisfied that $\widehat{\tau}\colon \widehat{AG}\rightarrow \widehat{Q}$ is the Lie algebroid of the Lie groupoid $\widehat{G}\rightrightarrows \widehat{Q}$.
\end{proof}
Next, suppose that a generalized discretization map ${\mathcal R}\colon AG \rightarrow G$ is $H$-equivariant, that is, for ${\mathbf h}\in H$ the following diagram is commutative for the corresponding actions:
\begin{equation*}
\xymatrix{
	 AG\ar[rr]^{\mathcal R}  \ar[d]_{A \Phi_{\mathbf h}}   && G \ar[d]^{\Phi_{\mathbf h}}  \\ 
	 AG \ar[rr]^{\mathcal R} &&  G}
\end{equation*} 

Then, it is clear that ${\mathcal R}$ induces a smooth map
$$\widehat{\mathcal R}\colon \widehat{AG}\rightarrow \widehat{G}$$
such that 
\begin{equation}\label{eq:Rreduced}\widehat{R}\circ \pi_{AG}=\pi_G\circ {\mathcal R}.\end{equation}
Moreover, sufficient conditions to obtain a generalized discretization map on the quotient spaces can be stated. 
\begin{theorem}
\label{thm:discreteQuotient}
    If ${\mathcal R}\colon AG \rightarrow G$ is a $H$-equivariant generalized discretization map on $G$, then $\widehat{\mathcal R}\colon \widehat{AG}\rightarrow \widehat{G}$ is a generalized discretization map on the quotient Lie groupoid $\widehat{G}\rightrightarrows \widehat{Q}$.
\end{theorem}
\begin{proof}
As ${\mathcal R}$ is a generalized discretization map, ${\mathcal R}\circ 0_{AG}=\epsilon$. It is also true
   that  $$\pi_{AG}\circ 0_{AG}=0_{\widehat{AG}}\circ \pi_Q, 
 \quad \pi_G\circ \epsilon=\widehat{\epsilon}\circ \pi_Q.$$
 Using Equation~\eqref{eq:Rreduced} and previous equalities, 
 $$\widehat{\mathcal R}\circ 0_{\widehat{AG}}\circ \pi_Q=\pi_G\circ {\mathcal R} \circ 0_{AG}= \pi_G\circ \epsilon=\widehat{\epsilon}\circ \pi_Q.$$
 Thus, $\widehat{R}\circ 0_{\widehat{AG}}=\widehat{\epsilon}$ and the first property in Definition~\ref{Def:newdiscretization} is satisfied.

Take $q\in Q$ and $a_q\in A_qG$ to prove the second property in Definition~\ref{Def:newdiscretization}. Using the properties in Definition~\ref{Def:newdiscretization} for ${\mathcal R}$ and the relations in Equations~\eqref{eq:Rreduced} and~\eqref{eq:structuralReduced} we compute:
 \begin{align*}
    \left(  T_{0_{\widehat{AG}}(\pi_Q(q))} \right.& \left.\widehat{\mathcal R}_{\pi_Q(q)} \right)\left( \pi_{AG}(a_q)\right)^{\rm v}_{\phi_{\widehat{AG}} (\pi_Q(q))}\\ =& 
T_{0_{\widehat{AG}}(\pi_Q(q)) } \widehat{\mathcal R}_{\pi_Q(q)} \left(\left. \dfrac{\rm d}{{\rm d}t}\right|_{t=0} t\pi_{AG}(a_q)\right)\\= & \left(T_{0_{\widehat{AG}}(\pi_Q(q))}\widehat{\mathcal R}_{\pi_Q(q)}\right) \left(\left. \dfrac{\rm d}{{\rm d}t}\right|_{t=0} \pi_{AG}(t\, a_q)\right)
     \\ =&\left(T_{0_{AG}(q)}\left(\widehat{\mathcal R}_{\pi_Q(q)}\circ \pi_{AG} \right)\right) \left(\left. \dfrac{\rm d}{{\rm d}t}\right|_{t=0} (t\, a_q)\right)\\ =& \left(T_{0_{AG}(q)}\left(\pi_G\circ {\mathcal R}_q\right)\right) \left(\left. \dfrac{\rm d}{{\rm d}t}\right|_{t=0} (t\, a_q)\right) \\=&\left(T_{\epsilon(q)}\pi_G\right) \left(a_q+\left(T_q\epsilon \circ T_{\epsilon(q)}\alpha \circ T_{0_{AG}(q)}{\mathcal R}_q\right) \left(\left. \dfrac{\rm d}{{\rm d}t}\right|_{t=0} (t\, a_q)\right) \right)\\ =& \pi_{AG}(a_q)+  \left(T_{0_{AG}(q)}\left(\pi_G \circ \epsilon \circ \alpha \circ {\mathcal R}_q\right)\right) \left(\left. \dfrac{\rm d}{{\rm d}t}\right|_{t=0} (t\, a_q)\right) \\ =& \pi_{AG}(a_q)+  \left(T_{0_{AG}(q)}\left(\widehat{\epsilon}\circ \widehat{\alpha}\circ \widehat{ \mathcal R}\circ \pi_{AG}\right)\right) \left(\left. \dfrac{\rm d}{{\rm d}t}\right|_{t=0} (t\, a_q)\right) \\ =& \pi_{AG}(a_q) \\+& \left(T_{\pi_Q(q)} \widehat{\epsilon}\circ T_{\widehat{\epsilon}(\pi_Q(q))}\widehat{\alpha}\circ T_{0_{\widehat{AG}}(\pi_Q(q))} \widehat{ \mathcal R}_{\pi_Q(q)}\right) 
    \left( \left.\dfrac{\rm d}{{\rm d}t}\right|_{t=0} t\pi_{AG}(a_q)\right)\\=&\pi_{AG}(a_q)\\ +&\left( T_{\pi_Q(q)}\widehat{\epsilon}\circ T_{\widehat{\epsilon}(\pi_Q(q))}\widehat{\alpha} \circ T_{0_{\widehat{AG}}(\pi_Q(q))}\widehat{\mathcal R}_{\pi_Q(q)} \right) \left( \pi_{AG}(a_q)\right)^{\rm v}_{0_{\widehat{AG}}(\pi_Q(q))} \, .
\end{align*} 
 The second property has been proved and the result follows.
\end{proof}
An application of Theorem~\ref{thm:discreteQuotient} is developed here for the pair groupoid case. If $\phi$ is a free and proper action of $H$ on $Q$, then $\widehat{Q}=Q/H$ is a quotient manifold. 
We have that the diagonal action $(\Phi=\phi\times \phi,\phi)$ of $H$ on $Q \times Q$ is a free and proper action of $H$ on the pair groupoid $Q\times Q\rightrightarrows Q$ by Lie groupoid isomorphisms. 

The corresponding action of $H$ on the Lie algebroid $\tau_Q\colon TQ \rightarrow Q$ is just the tangent lift of $\phi$ to $TQ$. We have the following commutative diagrams:

\begin{equation*}
\xymatrix{
H\times (Q\times Q)  \ar[d]<-4pt>_{\tilde{\rm pr}_1}  \ar[d]<4pt>^{\tilde{\rm pr}_2}\ar[rr]^{(\phi,\phi)} &&  Q\times Q \ar[d]<-4pt>_{{\rm pr}_1}  \ar[d]<4pt>^{{\rm pr}_2}&&   H\times (TQ)  \ar[d]_{\tilde{\tau}_Q} \ar[rr]^{T\phi} && TQ    \ar[d]_{\tau_Q}   \\  H\times Q  \ar[rr]^{\phi} &&  Q && H\times Q  \ar[rr]^{\phi}  &&   Q}
\end{equation*} 

The quotient objects 
\begin{eqnarray*}\widehat{Q\times Q}&=&(Q\times Q)/H \rightrightarrows \widehat{Q}=Q/H\, ,\\ \widehat{\tau}_Q\colon \widehat{A(Q\times Q)}&=&\widehat{TQ}=TQ/H \rightarrow \widehat{Q}=Q/H\end{eqnarray*}
are the Atiyah Lie groupoid and the corresponding Atiyah Lie algebroid, respectively (see Appendix~\ref{App:Atiyah}). 

A straightforward corollary of Theorem~\ref{thm:discreteQuotient} is the following result that connects with the contributions in~\cite{24IFAC}.

\begin{corollary}\label{Corol:reductionPairGroupoiod} If ${\mathcal R}\colon TQ \rightarrow Q\times Q$ is a $H$-equivariant discretization map on the pair groupoid, then $\widehat{{\mathcal R}}\colon \widehat{TQ}=TQ/H \rightarrow \widehat{Q\times Q}=(Q\times Q)/H$ is a generalized discretization map on the Atiyah groupoid $\widehat{Q\times Q}=(Q\times Q)/H.$
\end{corollary}

\section{Tangent lift of generalized discretization maps}\label{Sec:Lift}

Similar to the lifts of discretization maps developed in~\cite{21MBLDMdD}, we describe here the tangent lift of a generalized discretization map on the Lie groupoid to the tangent Lie groupoid.

In the definition in~\cite{21MBLDMdD} the canonical involution for the double tangent bundle is involved. In the Lie algebroid setting, it does not exist an appropriate version of the canonical involution, unless the prolongation of the Lie algebroid is considered, see~\cite{LeMaMa}. In this paper, we will not use the latter approach, but we propose another definition of the tangent lift of a discretization map which is isomorphic to the definition in~\cite{21MBLDMdD} because of Theorem~\ref{th:GtoG'}. As a consequence, the tangent lift of a generalized discretization map can be defined. 

Let us introduce such a new definition. Let $R_d\colon TQ \rightarrow Q\times Q$ be a discretization map, $\kappa_Q\colon TTQ \rightarrow TTQ$ and $TR_d\colon TTQ \rightarrow T(Q\times Q)$ be the canonical involution and the tangent lift of $R_d$, respectively, see~\cite{Godbillon} for more details.

There exists a canonical vector bundle isomorphism $$(T{\rm pr}_1,T{\rm pr}_2)\colon T(Q\times Q) \rightarrow TQ\times TQ$$
given by 
$$(T{\rm pr}_1,T{\rm pr}_2)(X_{(q_0,q_1)})=\left( T{\rm pr}_1(X_{(q_0,q_1)}),  T{\rm pr}_2(X_{(q_0,q_1)})\right), $$
for $X_{(q_0,q_1)}\in T_{(q_0,q_1)}(Q\times Q)$, where ${\rm pr}_i\colon Q\times Q\rightarrow Q$, $i=1,2$, are the canonical projections. The inverse morphism is denoted by $\Phi\colon TQ\times TQ \rightarrow T(Q\times Q)$. It was proved in~\cite{21MBLDMdD} that $R_d^T=(T{\rm pr}_1,T{\rm pr}_2)\circ TR_d\circ \kappa\colon TTQ \rightarrow TQ\times TQ$ is a discretization map. In other words, $R_d^T$ is a generalized discretization map for the pair groupoid $TQ\times TQ \rightrightarrows TQ$.

We will show now that the tangent lift $TR_d\colon TTQ \rightarrow T(Q\times Q)$
is a generalized discretization map for the tangent Lie groupoid $T(Q\times Q) \rightrightarrows TQ$ of the pair groupoid $Q\times Q \rightrightarrows Q$. The structural maps of the tangent Lie groupoid are the tangent lifts of the structural maps of the pair groupoid. Moreover, if $\tau_Q\colon TQ\rightarrow Q$ is the canonical projection, the vector bundle $T\tau_Q\colon TTQ\rightarrow TQ$ is just the Lie algebroid of the tangent Lie groupoid $T(Q\times Q)\rightrightarrows TQ$. In fact, the canonical involution $\kappa_Q$ is a Lie algebroid isomorphism between that Lie algebroid and the standard Lie algebroid given by the canonical projection $\tau_{TQ}\colon TTQ \rightarrow TQ$. In other words, the following diagram is commutative:
\begin{equation*}
\xymatrix{
	 TTQ\ar[rr]^{\kappa_Q}  \ar[dr]_{\tau_{TQ}}   && TTQ \ar[dl]^{T\tau_Q}  \\ 
	 & TQ & }
\end{equation*} 
Moreover, the following result about the inverse morphism $\Phi$ can be proved, see~\cite{Mackenzie}.

\begin{proposition}\label{Prop:PhiIso} The map $\Phi=(T{\rm pr}_1,T{\rm pr}_2)^{-1}\colon TQ\times TQ \rightarrow T(Q\times Q)$ is a Lie groupoid isomorphism of the pair groupoid $TQ\times TQ \rightrightarrows TQ$  and the tangent Lie groupoid $T(Q\times Q) \rightrightarrows TQ$ associated with the pair groupoid $Q\times Q \rightrightarrows Q$. Additionally, the Lie algebroid isomorphism 
\begin{equation*}
\xymatrix{
	 TTQ\ar[rr]^{A\Phi}  \ar[dr]_{\tau_{TQ}}   && TTQ \ar[dl]^{T\tau_Q}  \\ 
	 & TQ & }
\end{equation*} 
associated with $\Phi$ is the canonical involution, that is, $A\Phi=\kappa_Q$.
\end{proposition}
As a consequence, the tangent lift of a discretization map can be characterized as follows. 
\begin{corollary}\label{Corol:Tangent}
    If $R_d\colon TQ \rightarrow Q\times Q$ is a discretization map, then the tangent lift of $R_d$, $TR_d\colon TTQ \rightarrow T(Q\times Q)$, is a generalized discretization map on the tangent Lie groupoid $T(Q\times Q) \rightrightarrows TQ$, associated with the pair groupoid $Q\times Q\rightrightarrows Q$.

    Moreover, $TR_d$ is isomorphic to the discretization map $R_d^T\colon TTQ\rightarrow TQ\times TQ$ via the canonical involution $\kappa_Q\colon TTQ \rightarrow TTQ$ and the vector bundle isomorphism $(T{\rm pr}_1,T{\rm pr}_2)\colon T(Q\times Q)\rightarrow TQ \times TQ$. In particular, the following diagram is commutative:
    \begin{equation*}
\xymatrix{
	 TTQ\ar[rr]^{TR_d}  \ar[d]_{\kappa_Q}   && T(Q\times Q) \ar[d]^{(T{\rm pr}_1,T{\rm pr}_2)}  \\ 
	TTQ \ar[rr]^{R_d^T}  &&  TQ\times TQ }
\end{equation*} 
\end{corollary}

\begin{proof}
    From Theorem~\ref{th:GtoG'}, Proposition~\ref{Prop:PhiIso} and the fact that $R_d^T\colon TTQ \rightarrow TQ\times TQ$ is a discretization map, it follows that $R_d'=\Phi\circ R_d^T\circ A\Phi^{-1}\colon TTQ \rightarrow T(Q\times Q)$ is a generalized discretization map on the tangent Lie groupoid $T(Q\times Q)\rightrightarrows TQ$ associated with the pair groupoid $Q\times Q \rightrightarrows Q$.

    Moreover, Proposition~\ref{Prop:PhiIso} leads to $A\Phi^{-1}=\kappa_Q^{-1}=\kappa_Q$. Thus, $$R'_d=\Phi\circ(T{\rm pr}_1,T{\rm pr}_2) \circ TR_d \circ \kappa_Q\circ \kappa_Q=TR_d$$
    and the proof is concluded.
\end{proof}

Corollary~\ref{Corol:Tangent} is the key point to introduce the tangent lift of a generalized discretization on any Lie groupoid.

\subsection{General situation}

Let $G\stackrel[\beta]{\alpha}{\rightrightarrows}  Q$ be an arbitrary Lie groupoid, not necessarily the pair groupoid $Q\times Q{\rightrightarrows} Q$, with Lie algebroid $\tau\colon AG \rightarrow Q$. The tangent Lie groupoid $TG\stackrel[T\beta]{T\alpha}{\rightrightarrows} TQ$ has $T\epsilon\colon TQ \rightarrow TG$ as the identity section. The associated tangent Lie algebroid is given by $T\tau\colon T(AG)\rightarrow TQ$ that corresponds with $A(TG)$, see~\cite{Mackenzie}.

\begin{theorem}\label{Thm:tangentLift} If ${\mathcal R}\colon AG \rightarrow G$ is a generalized discretization map on $G$, then $T{\mathcal R}\colon T(AG)\rightarrow TG$ is a generalized discretization map on the tangent Lie groupoid $TG\stackrel[T\beta]{T\alpha}{\rightrightarrows} TQ$.
\end{theorem}

\begin{proof}
Since ${\mathcal R}$ is a generalized discretization map, the diagram $$ \xymatrix{AG \ar[rr]^{\mathcal R} && G \\ & Q \ar[ul]^{0_{AG}} \ar[ur]_{\epsilon}}$$ is commutative. Hence, 
$$ \xymatrix{T(AG) \ar[rr]^{T{\mathcal R}} && TG \\ & TQ \ar[ul]^{T0_{AG}} \ar[ur]_{T\epsilon}}$$ 
is also commutative. Note that $T0_{AG}\colon TQ \rightarrow T(AG)$ is the zero section of the tangent Lie algebroid $T(AG)\rightarrow TQ$.

The rest of the proof will be completed using the local coordinates as in Section~\ref{Sec:Local}. Remember that $(q^i,u^\gamma)$ are local coordinates on $G$ such that $\alpha(q,u)=q$ and $\epsilon(q)=(q,0)$, and $(q^i,y^\gamma)$ are induced local coordinates on $AG$. The local expression of ${\mathcal R}$ is
$${\mathcal R}(q,y)=({\mathcal R}^i(q,y),{\mathcal R}^\gamma(q,y))$$
and satisfies by Definition~\ref{Def:newdiscretization}: 
\begin{enumerate}
    \item ${\mathcal R}^i(q,0)=q^i$ and ${\mathcal R}^\gamma(q,0)=0^\gamma$ for all $i$, $\gamma$.
\item $\left( \left.\dfrac{\partial {\mathcal R}^\gamma}{\partial y^{\gamma'}} \right|_{(q,0)}\right)={\rm Id}$.
\end{enumerate}
The local coordinates for the tangent Lie groupoid $TG$ and the tangent Lie algebroid $T(AG)$ are respectively $(q^i,u^\gamma;\dot{q}^i,\dot{u}^\gamma)$ and $(q^i,y^\gamma;\dot{q}^i,\dot{y}^\gamma)$. The tangent lift of ${\mathcal R}$ has the following local expression:
\begin{align}\label{eq:TRlocal}
T{\mathcal R}(q^i,y^\gamma;\dot{q}^i,\dot{y}^\gamma)=& \left( {\mathcal R}^i(q,y), {\mathcal R}^\gamma(q,y), \dot{q}^j \, \left.\dfrac{\partial {\mathcal R}^i}{\partial q^j}\right|_{(q,y)}+ \dot{y}^\gamma \, \left.\dfrac{\partial {\mathcal R}^i}{\partial y^\gamma}\right|_{(q,y)}, \right. \nonumber \\ &  \left.\dot{q}^j \, \left.\dfrac{\partial {\mathcal R}^\gamma}{\partial q^j}\right|_{(q,y)}+ \dot{y}^{\gamma'} \, \left.\dfrac{\partial {\mathcal R}^\gamma}{\partial y^{\gamma'}}\right|_{(q,y)}  \right)=\left({\mathcal R}^i, {\mathcal R}^\gamma, (T{\mathcal R})^i, (T{\mathcal R})^\gamma\right) \, .
\end{align}
Note that the zero section of the tangent Lie algebroid $T\tau\colon T(AG)\rightarrow TQ$,  $T0_{AG}\colon TQ\rightarrow T(AG)$,  in local coordinates $(q,\dot{q})$ for $TQ$ is:
$$T_q0_{AG}(q,\dot{q})=(q,0;\dot{q},0).$$
Moreover, 
\begin{eqnarray*}
    T\alpha (q,y;\dot{q},\dot{y})&=&(q,\dot{q})\, ,\\
    T\epsilon (q,\dot{q})&=&(q,0;\dot{q},0)\, .
\end{eqnarray*}
The evaluation of Equation~\eqref{eq:TRlocal} at the point $(q,0;\dot{q},0)$ proves that $T{\mathcal R}(q,0;\dot{q},0)=(q,0;\dot{q},0)$ because 
\begin{eqnarray*}
    {\mathcal R}^i(q,0)&=&q^i,\\
     {\mathcal R}^\gamma(q,0)&=&0^\gamma,\\ 
      T{\mathcal R}^i(q,0;\dot{q},0)&=&\dot{q}^j \left.\dfrac{\partial {\mathcal R}^i}{\partial q^j}\right|_{(q,0)}=\dot{q}^i,\\
      T{\mathcal R}^\gamma(q,0;\dot{q},0)&=&\dot{q}^j \left.\dfrac{\partial {\mathcal R}^\gamma}{\partial q^j}\right|_{(q,0)}=0^\gamma\, .\\
\end{eqnarray*}

On the other hand, we must also prove that $\left( \left.\dfrac{\partial ({\mathcal R}^\gamma, T{\mathcal R}^{\bar{\gamma}})}{\partial (y^{\gamma'},\dot{y}^{\bar{\gamma}'})} \right|_{(q,0;\dot{q},0)}\right)={\rm Id}$. It follows because
\begin{eqnarray*}
\left.\dfrac{\partial {\mathcal R}^\gamma}{\partial y^{\gamma'}}\right|_{(q,0;\dot{q},0)}&=& \left.\dfrac{\partial {\mathcal R}^\gamma}{\partial y^{\gamma'}}\right|_{(q,0)}=\delta_{\gamma\gamma'}\, ,\\
\left.\dfrac{\partial {\mathcal R}^\gamma}{\partial \dot{y}^{\gamma'}}\right|_{(q,0;\dot{q},0)}&=& \left.\dfrac{\partial {\mathcal R}^\gamma}{\partial \dot{y}^{\gamma'}}\right|_{(q,0)}\equiv 0\, ,\\
\left.\dfrac{\partial (T{\mathcal R})^{\bar{\gamma}}}{\partial y^{\gamma'}}\right|_{(q,0;\dot{q},0)}&=& \dot{q}^j \stackrel{0}{\overbrace{\left.\dfrac{\partial^2 {\mathcal R}^{\bar{\gamma}}}{\partial q^j\partial y^{\gamma'}}\right|_{(q,0)}}}+ \stackrel{0}{\overbrace{\dot{y}^{\gamma''}}} \, \left.\dfrac{\partial^2{\mathcal R}^{\bar{\gamma}}}{\partial y^{\gamma''}\partial y^{\gamma'}}\right|_{(q,0)}\equiv 0\, ,\\
\left.\dfrac{\partial (T{\mathcal R})^{\bar{\gamma}}}{\partial \dot{y}^{\gamma'}}\right|_{(q,0;\dot{q},0)}&=& \left. \dfrac{\partial {\mathcal R}^{\bar{\gamma}}}{\partial  y^{\gamma'}}\right|_{(q,0)}=\delta_{\bar{\gamma} \gamma'}\,.
\end{eqnarray*}
\end{proof}

\begin{example}\label{Ex:TangentLieGroup} Going back to Example~\ref{Ex:LieAlgebra} where $G$ is a Lie group, the tangent lift of the generalized discretization map ${\mathcal R}\colon AG=\mathfrak{g}\rightarrow G$ makes the following diagram commutative:
$$ \xymatrix{T(AG)\simeq T(\mathfrak{g}) \ar[rr]^{T{\mathcal R}} && TG\\ & 0_{\rm e} \ar[ul]^{T0_{0_{\rm e}}} \ar[ur]_{T\epsilon}}$$ 
Note that $T\epsilon(0_{\rm e})=(e, 0)$ and $T{\mathcal R}(T0_{0_{\rm e}} (0_{\rm e}))=T{\mathcal R}(0_{\rm e},0)=({\mathcal R}(0_{\rm e}), 0)=(e,0)$.

As $G$ is a Lie group, the tangent bundle $TG$ is also a Lie group isomorphic to the semidirect product $G \ltimes \mathfrak{g}$ associated to the adjoint action of $G$ on $\mathfrak{g}$.

An example of discretization map on a Lie group is the exponential map, ${\mathcal R}={\rm exp}$ that leads to the following tangent lift of the generalized discretization map:
$$T{\mathcal R}(\xi,\eta)=\left({\rm exp}(\xi), T_\xi {\rm exp} (\eta) \right).$$

\end{example}

Finally, we mention that potential applications of Theorem~\ref{Thm:tangentLift} to obtain geometric integrators for symmetric second order differential equations on Lie groups will be described in Section~\ref{Sec:conclusion}.

\section{Cotangent lift of generalized dis\-cre\-ti\-za\-tion maps}\label{Sec:CoLift}

In~\cite{21MBLDMdD}, the cotangent lift of a  discretization map was introduced by using the canonical isomorphism between the symplectic manifolds $(T^*TQ,\omega_{TQ})$ and $(TT^*Q,{\rm d}_T\omega_Q)$, see~\cite{TuHamilton}. In the Lie algebroid setting, there is no an appropriate version of such an isomorphism, unless the prolongations of the Lie algebroid are considered~\cite{LeMaMa}. As in Section~\ref{Sec:Lift}, prolongations will not be used to define the cotangent lift of generalized discretization maps. An alternative definition of the cotangent lift of a discretization map is introduced and it is isomorphic to the definition in~\cite{21MBLDMdD} because of Theorem~\ref{th:GtoG'}. As a consequence, the cotangent lift of a generalized discretization map can be defined.  

Let us provide the new and equivalent definition of the cotangent lift of a discretization map. Let $R_d\colon TQ \rightarrow Q\times Q$ be a discretization map, $\alpha_Q\colon T^*TQ\rightarrow TT^*Q$ be the Tulczyjew isomorphism~\cite{TuHamilton}, $T^*R_d\colon T^*TQ\rightarrow T^*(Q\times Q)$ be the cotangent lift of $R_d$ and $\Phi^*\colon T^*(Q\times Q)\rightarrow T^*Q\times T^*Q$ be the vector bundle isomorphism given by
\begin{equation}\label{eq:PhiCotangent}
\Phi^*((\gamma,\delta)_{q_0,q_1})=(-\gamma_{q_0},\delta_{q_1}) \quad \mbox{for } \gamma_{q_0}\in T_{q_0}^*Q \; \mbox{and} \; \delta_{q_1}\in T_{q_1}^*Q.
\end{equation}
In~\cite{21MBLDMdD} it was proved that 
$$R^{T^*}=\Phi^*\circ T^*R_d\circ \alpha_Q^{-1}\colon TT^*Q\rightarrow T^*Q\times T^*Q $$
is a discretization map on the the pair groupoid $T^*Q\times T^*Q\rightrightarrows T^*Q$.

In order to prove below that the composition
$$T^*R_d \circ {\mathcal I}_{TQ}\colon T^*T^*Q\rightarrow T^*(Q\times Q)$$
is a generalized discretization map on the cotangent Lie groupoid $T^*(Q\times Q)\rightrightarrows T^*Q$, associated with the pair groupoid $Q\times Q\rightrightarrows Q$, it is necessary to use
the canonical antisymplectomorphism ${\mathcal I}_{TQ}\colon T^*T^*Q\rightarrow T^*TQ$  between the symplectic manifolds $(T^*T^*Q,\omega_{T^*Q})$ and $(T^*TQ,\omega_{TQ})$ given by 
$${\mathcal I}_{TQ}=\alpha_Q^{-1}\circ \Lambda_Q^\sharp,$$
where $\Lambda_Q^\sharp\colon T^*T^*Q\rightarrow TT^*Q$ is the vector bundle isomorphism associated with the canonical Poisson 2-vector $\Lambda_Q$ on $T^*Q$.

Moreover, the structural maps $\tilde{\alpha}$, $\tilde{\beta}$, $\tilde{\epsilon}$, $\tilde{m}$ and $\tilde{i}$ of the cotangent Lie groupoid $T^*(Q\times Q)\rightrightarrows T^*Q$ are defined as follows
   \begin{equation}\label{eq_structmapsCotang}
   \begin{array}{ll}
    \tilde{\alpha}((\gamma,\delta)_{q_0,q_1})=-\gamma_{q_0}, &  \tilde{\beta}((\gamma,\delta)_{q_0,q_1})=\delta_{q_1},\\
    \tilde{\epsilon} (\gamma_q)=(-\gamma_q,\gamma_q),  & \tilde{i}((\gamma,\delta)_{q_0,q_1})=(-\delta_{q_1},\gamma_{q_0}), \end{array} 
        \end{equation}\begin{equation*}
    \tilde{m}\left((\gamma,\delta)_{q_0,q_1},(-\delta,\mu)_{q_1,q_2} \right)=(\gamma_{q_0},\mu_{q_2}).      
    \end{equation*}
The vector bundle $\pi_{T^*Q}\colon T^*T^*Q\rightarrow T^*Q$ is the Lie algebroid of the cotangent Lie groupoid $T^*(Q\times Q)\rightrightarrows T^*Q$. In fact, the map $\Lambda_Q^\sharp\colon T^*T^*Q\rightarrow TT^*Q$ is a Lie algebroid isomorphism between $A(T^*(Q\times Q))$ and the Lie algebroid of the pair groupoid $T^*Q\times T^*Q \rightrightarrows T^*Q$. The following diagram summarizes this situation:
\begin{equation*}
    \xymatrix{ T^*T^*Q \ar[rr]^{\Lambda_Q^\sharp}  \ar[dr]_{\pi_{T^*Q}}  &&  TT^*Q \ar[dl]^{\tau_{T^*Q}} \\ & T^*Q}
\end{equation*}
More details can be found at~\cite{coste}.

\begin{proposition}\label{Prop:NewCotLift}
The map $\Phi^*\colon T^*(Q\times Q)\rightarrow T^*Q\times T^*Q$ is a Lie groupoid isomorphism of the cotangent Lie groupoid $T^*(Q\times Q)\rightrightarrows T^*Q$ associated with the pair groupoid $Q\times Q\rightrightarrows Q$ and the pair groupoid $T^*Q\times T^*Q\rightrightarrows T^*Q$. Moreover, the Lie algebroid isomorphism associated with $\Phi^*$ is just $\Lambda_Q^\sharp$:
\begin{equation*}
    \xymatrix{ T^*(T^*Q)\simeq A(T^*(Q\times Q)) \ar[rr]^{A\Phi^*=\Lambda_Q^\sharp}  \ar[dr]_{\pi_{T^*Q}}  &&  A(T^*Q\times T^*Q)\simeq TT^*Q \ar[dl]^{\tau_{T^*Q}} \\ & T^*Q}
\end{equation*}
\end{proposition}

\begin{proof}
    A direct computation using~\eqref{eq:PhiCotangent} and~\eqref{eq_structmapsCotang} shows that $\Phi^*$ is a Lie groupoid isomorphism. 

    The second part follows from
    $$(A\Phi^*)(0_{-\gamma_q},X_{\gamma_q})=(0_{\gamma_q},X_{\gamma_q})$$
    for $(0_{-\gamma_q},X_{\gamma_q})\in A_{\gamma_q}(T^*(Q\times Q))$ and the identification of $T^*(T^*Q)$ and $TT^*Q$ with $ A(T^*(Q\times Q))$ and $A(T^*Q\times T^*Q)$, respectively (see Appendices~\ref{App:LieAlgCotangGroup} and~\ref{App:ExamplePairGroup}).
\end{proof}

Now, the new characterization of the cotangent lift of a discretization map is provided.

\begin{corollary}\label{Corol:NewDefCotangLift}
Let $R_d\colon TQ\rightarrow Q\times Q$ be a discretization map. The map $\widetilde{\mathcal R}=T^*R_d\circ {\mathcal I}_{TQ}\colon T^*T^*Q\rightarrow T^*(Q\times Q)$ is a generalized discretization map on the cotangent Lie groupoid $T^*(Q\times Q)\rightrightarrows T^*Q$ associated with the pair groupoid $Q\times Q\rightrightarrows Q$.  

Moreover, $\widetilde{\mathcal R}$ is isomorphic to the discretization map $R^{T^*}\colon TT^*Q\rightarrow T^*Q\times T^*Q$ via the Lie algebroid isomorphism $\Lambda_Q^\sharp\colon T^*T^*Q\rightarrow TT^*Q$ and the Lie groupoid isomorphism $\Phi^*\colon T^*(Q\times Q) \rightarrow T^*Q\times T^*Q$. In particular, the following diagram is commutative:
\begin{equation*}
    \xymatrix{ T^*(T^*Q)\ar[r]^{\widetilde{\mathcal R}}  \ar[d]_{A(\Phi^*)=\Lambda_Q^\sharp}  &  T^*(Q\times Q) \ar[d]^{\Phi^*} \\ TT^*Q \ar[r]^{{\mathcal R}^{T^*}} & T^*Q\times T^*Q}
\end{equation*}
\end{corollary}
\begin{proof}
    From Theorem~\ref{th:GtoG'}, Proposition~\ref{Prop:NewCotLift} and the fact that ${\mathcal R}^{T^*}\colon TT^*Q\rightarrow T^*Q\times T^*Q$ is a discretization map, it follows that $$ \widetilde{\mathcal R}=\left(\Phi^*\right)^{-1}\circ {\mathcal R}^{T^*} \circ A(\Phi^*)\colon T^*T^*Q \rightarrow T^*(Q\times Q)$$
    is a generalized discretization map on the cotangent Lie groupoid $T^*(Q\times Q)\rightrightarrows T^*Q$ associated with the pair Lie groupoid $Q\times Q\rightrightarrows Q$.

    Moreover, Proposition~\ref{Prop:NewCotLift} guarantees that $A(\Phi^*)=\Lambda_Q^\sharp$. Thus, we conclude that
    $$  \widetilde{\mathcal R}=\left((\Phi^*)^{-1} \circ \Phi^*\right)\circ T^*R_d\circ \left(\alpha_Q^{-1} \circ \Lambda_Q^\sharp\right)=T^*R_d\circ {\mathcal I}_{T^*Q}.$$
\end{proof}
Corollary~\ref{Corol:NewDefCotangLift} is the key point to define the cotangent lift of a generalized discretization map as done in the following section.

\subsection{General situation} \label{Subsec:GeneralCotangentLift}

Let $G\stackrel[\beta]{\alpha}{\rightrightarrows} Q$ be a Lie groupoid with Lie algebroid $\tau\colon AG \rightarrow Q$.  The cotangent Lie groupoid \cite{coste} is given by $T^*G\stackrel[\tilde{\beta}]{\tilde{\alpha}}{\rightrightarrows} A^*G$ and it is associated with the Lie algebroid $\pi_{A^*G}\colon T^*(A^*G)\rightarrow A^*G$.
The structural maps of the cotangent groupoid are defined as follows (see Appendix~\ref{App:SymplGroupoid}). For a section $X\colon Q\rightarrow AG$ of the Lie algebroid, the source map $\tilde{\alpha}$, defines for every $\gamma_g\in T^*G$ the element $\tilde{\alpha}(\gamma_g)\in A^*_{\alpha(g)}(G)$ given by 
\begin{equation}\label{eq:sourcecotangent}\langle \tilde{\alpha}(\gamma_g), X(\alpha(g))\rangle=\langle \gamma_g,\overrightarrow{X}(g)\rangle\, ,\qquad g\in G, \; \gamma_g\in T^*G,\end{equation}
where $\overrightarrow{X}$ is the right-invariant vector field on $G$ associated with the section $X\colon Q\rightarrow AG$ (see Appendix~\ref{App:LieAlgForGroup}). 

The target map $\tilde{\beta}$ for every $\gamma_g\in T^*_gG$ defines $\tilde{\beta}(\gamma_g)$ in $A^*_{\beta(q)}G$ as
\begin{equation}\label{eq:targetcotangent}\langle \tilde{\beta}(\gamma_g), X(\beta(g))\rangle=\langle \gamma_g,\overleftarrow{X}(g)\rangle\quad \mbox{for } X\in \Gamma(AG)\, ,\end{equation}
where $\overleftarrow{X}$ is the left-invariant vector field on $G$ associated with $X$ (see Appendix~\ref{App:LieAlgForGroup}).

The identity map sends every $\alpha_q$ in $A^*_qG$ to $\tilde{\epsilon}(\alpha_q)$ in $T^*_{\epsilon(q)}G$ given by
\begin{equation}\label{eq:idcotangentlift} \tilde{\epsilon}(\alpha_q)=\left( {\rm id}-T_q\epsilon \circ T_{\epsilon(q)}\alpha\right)^*(\alpha_q)\, , \end{equation}
where $\left( {\rm id}-T_q\epsilon \circ T_{\epsilon(q)}\alpha\right)\colon T_{\epsilon(q)}G\rightarrow A_qG$ is the projection defined by 
$$\left( {\rm id}-T_q\epsilon \circ T_{\epsilon(q)}\alpha\right)(v_{\epsilon(q)})=v_{\epsilon(q)}-T_q\epsilon \left(T_{\epsilon(q)}\alpha(v_{\epsilon(q)})\right)\in V_{\epsilon(q)}\alpha=A_qG.$$

Moreover, $A(T^*G)$ is isomorphic to $T^*(A^*G)$ because of the existence of the following isomorphism:
\begin{eqnarray}\label{eq:isocotangentlift} 
    \Psi\colon T^*_{\alpha_q}(A^*G)&\longrightarrow & A_{\tilde{\epsilon}(\alpha_q)}(T^*G)=V_{\tilde{\epsilon}(\alpha_q)}\tilde{\alpha} \subseteq T_{\tilde{\epsilon}(\alpha_q)}(T^*G) \\ w_{\alpha_q} &\longmapsto & \Lambda^\sharp_G \left( \left(T_{\tilde{\epsilon}(\alpha_q)}^* \tilde{\beta} \right)(w_{\alpha_q})\right) \, , \nonumber
\end{eqnarray}
where $\Lambda_G$ is the canonical Poisson structure on $T^*G$, that is, 
$$i_{\Lambda^\sharp_G \left( \left(T_{\tilde{\epsilon}(\alpha_q)}^* \tilde{\beta} \right)(w_{\alpha_q})\right) } \omega_G\left(\tilde{\epsilon}(\alpha_q)\right)=-\left( T^*_{\tilde{\epsilon}(\alpha_q)}\tilde{\beta}\right)(w_{\alpha_q})\, ,$$
with $\omega_G$ the canonical symplectic structure of $T^*G$ (see Appendix~\ref{App:LieAlgCotangGroup}). 

Before introducing the definition of the cotangent lift of a generalized discretization map, let us recall the definition of the canonical antisymplectomorphism ${\mathcal I}_A\colon T^*A^*\rightarrow T^*A$ between the vector bundles $\pi_{A^*}\colon T^*A^*\rightarrow A^*$ and $T^*A\rightarrow A^*$ for an arbitrary vector bundle $\tau\colon A \rightarrow Q$. The local expression of ${\mathcal I}_A$ in local fibered coordinates on $T^*A^*$ and $T^*A$ is 
\begin{equation}\label{Eq:IA} {\mathcal I}_A(q^i,y_\alpha;\delta q^i,\delta y_\alpha)=(q^i,\delta y_\alpha;-\delta q^i,y_\alpha)\end{equation}
and ${\mathcal I}^*_A(\omega_A)=-\omega_{A^*}$, see~\cite{XuMac} for more details. Note that for the particular case $A=TQ$, ${\mathcal I}_A={\mathcal I}_{TQ}=\alpha_Q^{-1}\circ \Lambda_Q^\sharp$.

Let us prove now the main result of this section. 

\begin{theorem}\label{thm:CotangentLift}
    If ${\mathcal R}\colon AG \rightarrow G$ is a generalized discretization map on the Lie groupoid $G\rightrightarrows Q$, then $\widetilde{\mathcal R}= T^*{\mathcal R} \circ {\mathcal I}_{AG}\colon T^*(A^*G)\rightarrow T^*G$ is a generalized discretization map on the cotangent Lie groupoid $T^*G\rightrightarrows A^*G$.
\end{theorem}

The result has been proved at the beginning of the section for the particular case of the pair groupoid, that is, $G=Q\times Q\rightrightarrows Q$.

When $G$ is a Lie group with Lie algebra $\mathfrak{g}$, the cotangent Lie groupoid is $T^*G\rightrightarrows \mathfrak{g}^*$. For a generalized discretization map ${\mathcal R}\colon \mathfrak{g}\rightarrow G$ on $G$, we consider the cotangent lift of ${\mathcal R}$ as the map
$$\widetilde{\mathcal R}=T^*{\mathcal R}\circ {\mathcal I}_{\mathfrak{g}}\colon \mathfrak{g}^*\times \mathfrak{g} \rightarrow G\times \mathfrak{g}^*,$$
where the canonical identification $T^*\mathfrak{g}^*\simeq \mathfrak{g}^*\times \mathfrak{g}$ and the right trivialization $T^*G\simeq G\times \mathfrak{g}^*$ have been used. 

Now, under the identification $T^*\mathfrak{g}\simeq \mathfrak{g}\times \mathfrak{g}^*$, it follows that the map ${\mathcal I}_{\mathfrak{g}}\colon \mathfrak{g}^*\times \mathfrak{g} \rightarrow   \mathfrak{g}\times \mathfrak{g}^*$ is given by 
$${\mathcal I}_{\mathfrak{g}}(\mu,\xi)=(\xi,\mu). $$ Thus, using~\eqref{eq:sourcecotangent} we have that for every $(\mu,\xi)\in \mathfrak{g}^*\times \mathfrak{g}$, 
$$\widetilde{\mathcal R} (\mu,\xi)=\left( {\mathcal R}(\xi),\tilde{\alpha}((T_{{\mathcal R}(\xi)}{\mathcal R}^{-1})^*(\mu))\right)\, .$$

Let us prove that the map $\widetilde{\mathcal R} $ satisfies both properties in Definition~\ref{Def:newdiscretization}. 
\begin{enumerate}
    \item Under the identifications $T^*\mathfrak{g}^*\simeq \mathfrak{g}^*\times \mathfrak{g}$ and $T^*G\simeq G\times \mathfrak{g}^*$, we have that $\tilde{\epsilon}(\mu)=({\rm e},\mu)$, where $e$ is the identity element of the Lie group $G$. It must be shown that
    $$\widetilde{\mathcal R}  (\mu,0_{\mathfrak{g}})= ({\rm e},\mu) \quad \forall \;\mu\in \mathfrak{g}^*.$$
That follows from ${\mathcal R}(0_{\mathfrak{g}})={\rm e}$ and the fact that $T_{0_{\mathfrak{g}}}{\mathcal R}\colon T_{0_{\mathfrak{g}}}\mathfrak{g}\simeq \mathfrak{g} \rightarrow T_{\rm e}G\simeq \mathfrak{g} $ is the identity map. Thus, $T_{\rm e}{\mathcal R}^{-1}\colon\mathfrak{g} \rightarrow T_{0_{\mathfrak{g}}}\mathfrak{g} \simeq \mathfrak{g}$ and  $\left(T_{\rm e}{\mathcal R}^{-1}\right)^*\colon\mathfrak{g}^* \rightarrow \mathfrak{g}^*$  are also identity maps. 
    \item Under the right trivialization $G\times \mathfrak{g}^*\rightrightarrows \mathfrak{g}^*$ of the cotangent Lie groupoid $T^*G \rightrightarrows \mathfrak{g}^* $, we have that the source of the Lie groupoid $G\times \mathfrak{g}^* \rightrightarrows \mathfrak{g}^*$ is just the canonical projection on the second factor ${\rm pr}_2\colon G\times \mathfrak{g}^*\rightarrow \mathfrak{g}^*$ and the identity section $\tilde{\epsilon}$ is the map
    $$\tilde{\epsilon}=(C_{\rm e},{\rm id})\colon    \mathfrak{g}^* \rightarrow G \times  \mathfrak{g}^*, \quad \tilde{\epsilon}(\mu)=({\rm e},\mu)\in  G \times  \mathfrak{g}^*. $$
    In addition, the fiber of the Lie algebroid of $G\times \mathfrak{g}^*\rightrightarrows \mathfrak{g}^* $ at $\mu\in \mathfrak{g}^*$ is $$A_\mu(G\times \mathfrak{g}^*)\simeq \mathfrak{g}\times \{0_{T_\mu \mathfrak{g}^*}\}$$
    and
    $$ T_\mu\tilde{\epsilon}(T_\mu \mathfrak{g}^*)=T_\mu\left(C_{\rm e},{\rm id} \right)(T_\mu \mathfrak{g}^*)=\{0_{\mathfrak{g}}\}\times T_\mu \mathfrak{g}^*. $$
    Moreover, using the canonical identification $\mathfrak{g}^*\times \mathfrak{g}\simeq T^*\mathfrak{g}^*$, we have that for every $(\xi,0_{T_\mu\mathfrak{g}^*})\in A_\mu (G\times \mathfrak{g}^*)\simeq \mathfrak{g}$,
    $$ (\xi,0_{T_\mu\mathfrak{g}^*})^{\mathtt{v}}_{(\mu,0)}=\left.\dfrac{\rm d}{\rm dt}\right|_{t=0}(\mu,t\xi).$$
    As stated in the second item in Definition~\ref{Def:newdiscretization}, we must prove that 
    \begin{align*}
    \left. \dfrac{\rm d}{{\rm d}t} \right|_{t=0} \left( \widetilde{\mathcal R}(\mu,t\xi)\right)& =\left. \dfrac{\rm d}{{\rm d}t} \right|_{t=0} \left( {\mathcal R}(t\xi), \tilde{\alpha}\left(  T_{{\mathcal R}(t\xi)} {\mathcal R}^{-1}\right)^*(\mu)\right)\\ &= \left(\xi, \left. \dfrac{\rm d}{{\rm d}t} \right|_{t=0}  \tilde{\alpha}\left(  T_{{\mathcal R}(t\xi)} {\mathcal R}^{-1}\right)^*(\mu)\right).
    \end{align*}
    It follows using that ${\mathcal R}(0_{\mathfrak{g}})={\rm e}$ and the fact that $T_{0_{\mathfrak{g}}} {\mathcal R}\colon T_{0_{\mathfrak{g}}} \mathfrak{g}\simeq \mathfrak{g}\rightarrow T_{\rm e}G\simeq \mathfrak{g}$ is just the identity map.
   \end{enumerate}

This concludes the proof for the Lie group case. Let us prove Theorem~\ref{thm:CotangentLift} in general.

\begin{proof}[Proof of Theorem~\ref{thm:CotangentLift}]
The map $\widetilde{\mathcal R}$ is a generalized discretization map if it satisfies the following two properties adapted from Definition~\ref{Def:newdiscretization}:
\begin{enumerate}
    \item $\widetilde{\mathcal R}\circ 0_{T^*(A^*G)}=\tilde{\epsilon}$, where $0_{T^*(A^*G)}\colon A^*G \rightarrow T^*(A^*G)$ is the zero 1-form on $A^*G$ and $\tilde{\epsilon}\colon A^*G \rightarrow T^*G$ is the identity section of the Lie groupoid $T^*G\rightrightarrows A^*G$.

    \item For every $\alpha_q\in A^*_qG$ and $\gamma_{\alpha_q}\in T^*_{\alpha_q}(A^*G)$,
    \begin{align*} \left( T_{0_{T^*(A^*G)(\alpha_q)}}(\widetilde{\mathcal R}) \right) &(\gamma_{\alpha_q})^{\rm v}_{0_{T^*(A^*G)(\alpha_q)}}=\Psi (\gamma_{\alpha_q})\\+&\left[  T_{\alpha_q}\tilde{\epsilon} \circ T_{\tilde{\epsilon}(\alpha_q)} \tilde{\alpha} \circ T_{0_{T^*(A^*G)(\alpha_q)}}\widetilde{\mathcal R}\right](\gamma_{\alpha_q})^{\rm v}_{0_{T^*(A^*G)(\alpha_q)}}\, .\end{align*}
\end{enumerate}

As $ {\mathcal I}_{AG}\colon T^*(A^*G)\rightarrow T^*(AG)$ is a vector bundle morphism over $A^*G$, that is,
\begin{equation*}
    \xymatrix{ T^*(A^*G) \ar[rr]^{{\mathcal I}_{AG}} \ar[dr]_{\pi_{A^*G}} && T^*(AG) \ar[dl]^{\mathtt{v}^*} \\ & A^*G &}
\end{equation*} it is satisfied that 
$$ {\mathcal I}_{AG}\circ 0_{T^*(A^*G)}=0_{T^*(AG)}\, ,$$
where $0_{T^*(AG)}\colon A^*G \rightarrow T^*(AG)$ is the zero section of the vector bundle $\mathtt{v}^*\colon T^*(AG)\rightarrow A^*G$, that is, for $\alpha_q\ A_q^*G$
$$ 0_{T^*(A^*G)}(\alpha_q)\in T^*_{0_{AG}(q)}(AG).$$
For $X_{0_{AG}(q)}\in T_{0_{AG}(q)} (AG)$, we have
$$X_{0_{AG}(q)}=(a_q)^{\rm v}_{0_{AG}(q)}+(T_q0_{AG})(v_q), $$ where $a_q\in A_qG$ and $v_q\in T_qQ$, also
$$\langle 0_{T^*(AG)}(\alpha_q), X_{0_{AG}(q)} \rangle=\langle \alpha_q,a_q\rangle\, .$$
In addition, it follows that
$$T^*\widetilde{\mathcal R}(0_{T^*(AG)}(\alpha_q))\in T^*_{{\mathcal R}(0_{AG}(q)) }G=T^*_{\epsilon(q)}G$$
because ${\mathcal R}$ is a generalized discretization map. 

Moreover, if $X_{\epsilon(q)}\in T_{\epsilon(q)}G$, then 
$$X_{\epsilon(q)}=a_q+(T_q\epsilon)(v_q),$$
with $a_q\in A_qG$ and $v_q\in T_qQ$. We deduce that
\begin{align*}
    \langle T^*{\mathcal R}(0_{T^*(AG)}(\alpha_q)) , & a_q+(T_q\epsilon)(v_q) \rangle= \langle 0_{T^*(AG)}(\alpha_q) , (T_{\epsilon(q)} {\mathcal R}^{-1})( a_q+(T_q\epsilon)(v_q)) \rangle \\ =& 
    \langle 0_{T^*(AG)}(\alpha_q) , (T_{\epsilon(q)} {\mathcal R}^{-1})( a_q)+T_q({\mathcal R}^{-1}\circ \epsilon)(v_q)) \rangle \\ =&   \langle 0_{T^*(AG)}(\alpha_q) ,  (a_q)^{\rm v}_{0_{AG}(q)}+T_q0_{AG}(v_q)\rangle = \langle \alpha_q,a_q\rangle\, .
\end{align*}

On the other hand, we have that
$$\langle \tilde{\epsilon}(\alpha_q), a_q+(T_q\epsilon)(v_q)\rangle=\langle \alpha_q,a_q+(T_q\epsilon)(v_q)-(T_q\epsilon)(v_q)\rangle=\langle \alpha_q, a_q\rangle\, .$$

The first property has been proved. Let us prove locally the second one. Take local coordinates as in Section~\ref{Sec:Local}:
\begin{itemize}
    \item $(q^i)$ local coordinates on $Q$;
    \item $(q^i,u^\gamma)$ local coordinates on $G$ such that $\alpha(q,u)=q$, $\epsilon(q)=(q,0)$;
    \item $\{e_\gamma=\partial/ \partial u^\gamma \circ \epsilon \}$ is a local basis of the sections $\Gamma(AG)$ of the vector bundle $AG\rightarrow Q$;
    \item $(q^i,y^\gamma)$ local coordinates on $AG$;
    \item $(q^i,y_\gamma)$ the dual local coordinates on $A^*G$.
\end{itemize}
The inverse map of ${\mathcal R}\colon AG\rightarrow G$ has as local expression:
$${\mathcal R}^{-1}(q,u)=\left( \left({\mathcal R}^{-1}\right)^i(q,u),\left({\mathcal R}^{-1}\right)^\gamma(q,u) \right)$$ such that for every $i$, $\gamma$:
\begin{enumerate}
    \item $\left({\mathcal R}^{-1}\right)^i(q,0)=q^i$, $\left({\mathcal R}^{-1}\right)^\gamma(q,0)=0^\gamma$;
    \item $\left( \left.\dfrac{\partial \left({\mathcal R}^{-1}\right)^\gamma}{\partial u^{\gamma'} }\right|_{(q,0)}\right)={\rm Id}$.
\end{enumerate}
The following local coordinates are necessary:
\begin{itemize}
    \item $(q^i,y_\gamma;\delta q^i, \delta y_\gamma)$ on $T^*(A^*G)$;
     \item $(q^i,y^\gamma;\Pi_i, \Pi_\gamma)$ on $T^*(AG)$;
     \item $(q^i,u^\gamma;p_i, p_\gamma)$ on $T^*G$;
          \item $(q^i,y_\gamma;\dot{q}^i,\dot{y}_\gamma)$ on $T(A^*G)$.
\end{itemize}
Following the framework developed in~\cite{2015JCDavidLocalDiscrete}, we obtain now the local expressions of $\tilde{\alpha}\colon T^*G \rightarrow A^*G$ and $\tilde{\beta}\colon T^*G\rightarrow A^*G$. Suppose that $\beta(q,u)=(\beta^i(q,u))$. If $g\equiv (q,u)$, $\bar{g} \equiv(\bar{q},\bar{u})$ are points of $G$, then the pair $(g,\bar{g})$ is composable, that is, it lives in $ G_2$,
 if and only if $\beta^i(q,u)=\bar{q}^i$, for all $i$. Then the multiplication is defined as $$g\cdot \bar{g}=(q,p(q,u,\bar{u})).$$
 Let \begin{eqnarray*}
     \rho^i_\gamma(q)&=&\left.\dfrac{\partial \beta^i}{\partial u^\gamma} \right|_{(q,0)}\, , \\
     L^\gamma_\mu(q,u)&=&\left.\dfrac{\partial p^\gamma}{\partial \bar{u}^\mu} \right|_{(q,u,0)}\, , \\
      R^\gamma_\mu(q,\bar{u})&=&\left.\dfrac{\partial p^\gamma}{\partial u^\mu} \right|_{(q,0,\bar{u})}\, ,
     \end{eqnarray*}
    the left and right-invariant vector fields are, respectively,
    \begin{eqnarray}
     \overleftarrow{e_\gamma}(q,u)&=& L^\mu_\gamma(q,u) \, \left. \dfrac{\partial }{\partial u^\mu}\right|_{(q,u)} \, , \label{eq:left}\\
     \overrightarrow{e_\gamma}(q,u)&=&-\rho^i_\gamma(q) \, \left. \dfrac{\partial }{\partial q^i} \right|_{(q,u)} + R^\mu_\gamma (q,u) \, \left. \dfrac{\partial }{\partial u^ \mu} \right|_{(q,u)} \, , \label{eq:right}
     \end{eqnarray}
     see~\cite{2015JCDavidLocalDiscrete}. Thus, using~\eqref{eq:sourcecotangent},~\eqref{eq:targetcotangent},~\eqref{eq:idcotangentlift} ,~\eqref{eq:left} and~\eqref{eq:right}, we deduce that
     \begin{eqnarray} \tilde{\alpha}(q^i,u^\gamma, p_i,p_\gamma)&=&\left(q^i,-\rho^i_\mu(q)p_i+R^\gamma_\mu(q,u) \, p_\gamma \right)\, , \label{eq:alpha}\\
      \tilde{\beta}(q^i,u^\gamma, p_i,p_\gamma)&=&\left(\beta^i(q,u),L^\gamma_\mu(q,u) p_\gamma\right) \, ,  \label{eq:beta} \\
      \widetilde{\epsilon}(q^i,y_\gamma)&=&(q^i,0;0,y_\gamma). \label{eq:epsilon}
     \end{eqnarray}
    On the other hand, the local expression of the isomorphism $\Lambda^\sharp_G\colon T^*(T^*G)\rightarrow T(T^*G)$ is 
     \begin{equation*}
         \begin{array}{lcl}
        \Lambda^\sharp_G({\rm d}q^i)=-\dfrac{\partial}{\partial p_i},  &&  \Lambda^\sharp_G({\rm d}u^\gamma)=-\dfrac{\partial}{\partial p_\gamma}, \\
         \Lambda^\sharp_G({\rm d}p_i)=\dfrac{\partial}{\partial q_i},  &&  \Lambda^\sharp_G({\rm d}p_\gamma)=-\dfrac{\partial}{\partial u^\gamma}.
         \end{array}
     \end{equation*}
     Equation~\eqref{eq:beta} and the equalities
          \begin{equation*}
         \begin{array}{lll}
         \left.\dfrac{\partial p^i}{\partial q^j}\right|_{(q,0)}=\delta^i_j, &  \left.\dfrac{\partial p^\gamma}{\partial \bar{u}^mu}\right|_{(q,0,0)}=\delta^\gamma_\mu, &
         \left.\dfrac{\partial^2 p^\gamma}{\partial q^i\partial \bar{u}^\mu}\right|_{(q,0,0)}=0,
         \end{array}
     \end{equation*}
     from~\cite{2015JCDavidLocalDiscrete}, lead to the following characterization of the isomorphism    
     $\Psi\colon T^*(A^*G)\rightarrow A(T^*G)=V_{\tilde{\epsilon}(A^*G)}\tilde{\alpha}$:
     \begin{eqnarray*}
         \Psi\left(\left.{\rm d}q^i\right|_{(q,y)}\right)&=& -\left.\dfrac{\partial}{\partial p_i}\right|_{(q,0;0,y)}-\rho^i_\gamma(q) \left.\dfrac{\partial}{\partial p_\gamma}\right|_{(q,0;0,y)}\, ,\\
          \Psi\left(\left.{\rm d}y_\mu\right|_{(q,y)}\right)&=& \left.\dfrac{\partial}{\partial u^\mu}\right|_{(q,0;0,y)}-\left.\dfrac{\partial^2 p^\gamma}{\partial u^{\gamma'} \partial \bar{u}^\mu}\right|_{(q,0,0)}\, y_\gamma\,  \left.\dfrac{\partial}{\partial p_{\gamma'}}\right|_{(q,0;0,y)}\, .
     \end{eqnarray*}
    Remember that $\widetilde{\mathcal R}=T^*{\mathcal R}\circ {\mathcal I}_{AG}\colon T^*(A^*G)\rightarrow T^*G$. Therefore, from~\eqref{Eq:IA} it follows that
    \begin{align}
    \widetilde{\mathcal R}&(q^i,y_\gamma;\delta q^i,\delta y_\gamma) = \left( {\mathcal R}^i(q,\delta y), {\mathcal R}^\gamma(q,\delta y); -\delta q^j\, \left. \dfrac{\partial \left({\mathcal R}^{-1}\right)^j}{\partial q^i} \right|_{{\mathcal R}(q,\delta y)} \right. \label{eq:Rtilde} \\  &  + \left. y_\gamma\, \left.\dfrac{\partial \left({\mathcal R}^{-1}\right)^\gamma}{\partial q^i} \right|_{{\mathcal R}(q,\delta y)} , -\delta q^j\, \left.\dfrac{\partial \left({\mathcal R}^{-1}\right)^j}{\partial u^\gamma} \right|_{{\mathcal R}(q,\delta y)}+   y_{\gamma'}\, \left.\dfrac{\partial \left({\mathcal R}^{-1}\right)^{\gamma'}}{\partial u^\gamma} \right|_{{\mathcal R}(q,\delta y)}
\right) \, .  \nonumber  \end{align}
From Equation~\eqref{eq:Rtilde} and $\left.\dfrac{\partial ({\mathcal R}^{-1})^i}{\partial q^j} \right|_{(q,0)}=\delta^i_j$,
we obtain that

\begin{align*}
    \left( T_{(q^i,y_\gamma;0,0)}\widetilde{\mathcal R}\right) & \left( \left. \dfrac{\partial}{\partial ( \delta q^i)}\right|_{(q^i,y_\gamma,0,0)}\right) \\ -&\left[(T_{(q^i,y_\gamma)}\tilde{\epsilon})\circ T_{(q^i,0;0,y_\gamma)}\tilde{\alpha}\circ T_{(q^i,y_\gamma,0,0)}\widetilde{\mathcal R}\right] \left( \left.\dfrac{\partial}{\partial  ( \delta q^i)} \right|_{(q^i,y_\gamma,0,0)}\right) \\ =&-\left.\dfrac{\partial}{\partial p_i}\right|_{(q^i,0;0,y_\gamma)}- \left.\dfrac{\partial ({\mathcal R}^{-1})^i}{\partial u^\gamma}\right|_{(q,0)}   \left.\dfrac{\partial}{\partial p_\gamma}\right|_{(q^i,0;0,y_\gamma)}\\&-\rho^i_\gamma(q) \left.\dfrac{\partial}{\partial p_\gamma}\right|_{(q^i,0;0,y_\gamma)}+ \left.\dfrac{\partial ({\mathcal R}^{-1})^i}{\partial u^\gamma}\right|_{(q,0)}   \left.\dfrac{\partial}{\partial p_\gamma}\right|_{(q^i,0;0,y_\gamma)}\\&=\Psi\left(\left.{\rm d}q^i\right|_{(q^i,y_\gamma)}\right)  .\end{align*}
Similarly, from 
$$ \left. \dfrac{\partial {\mathcal R}^\gamma}{\partial y^\mu}\right|_{(q,0)}=\delta^\gamma_\mu, \quad  \left. \dfrac{\partial^2 \left({\mathcal R}^{-1}\right)^\gamma }{\partial q^i\partial q^j}\right|_{(q,0)}=\left. \dfrac{\partial^2 \left({\mathcal R}^{-1}\right)^\gamma }{\partial u^\mu \partial q^i}\right|_{(q,0)}=0,
$$
it is obtained that
    \begin{align*}
\left( T_{(q^i,y_\gamma;0,0)} \widetilde{\mathcal R}\right) & \left( \left. \dfrac{\partial}{\partial ( \delta y_\mu)}\right|_{(q^i,y_\gamma,0,0)}\right) =  \left. \dfrac{\partial {\mathcal R}^i}{\partial y^\mu}\right|_{(q,0)}    \left.\dfrac{\partial}{\partial q^i}\right|_{(q^i,0;0,y_\gamma)}\\ + &   \left.\dfrac{\partial}{\partial u^\mu}\right|_{(q^i,0;0,y_\gamma)}+y_{\gamma'} \, \left. \dfrac{\partial^2 \left({\mathcal R}^{-1}\right)^{\gamma'} }{\partial u^\gamma\partial u^\mu}\right|_{(q,0)} \left.\dfrac{\partial}{\partial p_\gamma}\right|_{(q^i,0;0,y_\gamma)}.\end{align*}  

On the other hand, from~\eqref{eq:alpha},~\eqref{eq:epsilon} and 
$${\mathcal R}^\gamma_\mu(q,0)=\delta^\gamma_\mu, \quad \left. \dfrac{\partial {\mathcal R}^{\gamma'}_\gamma}{\partial q^i}\right|_{(q,0)}=0,$$
see~\cite{2015JCDavidLocalDiscrete}, as well as
$$\left. \dfrac{\partial {\mathcal R}^\gamma_{\gamma'}}{\partial \bar{u}^\mu}\right|_{(q,\bar{u})}= \left. \dfrac{\partial^2 p^\gamma}{\partial u^{\gamma'} \partial \bar{u}^\mu}\right|_{(q,0,\bar{u})},$$
it follows that
\begin{align*} T_{(q^i,y_\gamma;0,0)}&\left(\tilde{\epsilon}\circ \tilde{\alpha}\circ \tilde{\mathcal R}\right) \left( \left.\dfrac{\partial}{\partial (\delta y_\mu)} \right|_{(q^i,y_\gamma;0,0)} \right)=\left. \dfrac{\partial {\mathcal R}^{i}}{\partial y^\mu}\right|_{(q,0)} \left.\dfrac{\partial}{\partial q^i}\right|_{(q^i,0;0,y_\gamma)} \\ - &\left.\dfrac{\partial^2 p^\gamma}{\partial u^{\gamma'} \partial \bar{u}^\mu}\right|_{(q,0)}\, y_\gamma\,  \left.\dfrac{\partial}{\partial p_{\gamma'} }\right|_{(q^i,0;0,y_\gamma)}-y_{\gamma'}\, \left. \dfrac{\partial^2 \left({\mathcal R}^{-1}\right)^{\gamma'} }{\partial u^\mu \partial u^\gamma}\right|_{(q,0)}\ \left.\dfrac{\partial}{\partial p_{\gamma} }\right|_{(q^i,0;0,y_\gamma)}.\end{align*}

Thus, we conclude that
\begin{align*}\left( T_{(q^i,y_\gamma;0,0)} \widetilde{\mathcal R}\right) & \left( \left. \dfrac{\partial}{\partial ( \delta y_\mu)}\right|_{(q^i,y_\gamma,0,0)}\right)- T_{(q^i,y_\gamma;0,0)}\left(\tilde{\epsilon}\circ \tilde{\alpha}\circ \tilde{\mathcal R}\right)  \left( \left.\dfrac{\partial}{\partial (\delta y_\mu)} \right|_{(q^i,y_\gamma;0,0)} \right)\\=&\Psi\left(\left.{\rm d}y_\mu\right|_{(q^j,y_\gamma)}\right)  \, ,
\end{align*}
what ends the proof.
\end{proof}

\begin{example}\label{Ex:Cotangent} 
A discretization map ${\mathcal R}\colon TQ \rightarrow Q\times Q$ and the generalized discretization map $\tau\colon \mathfrak{g}\rightarrow G$ in Example~\ref{Ex:TangentLieGroup} define the generalized discretization map $(\tau,{\mathcal R})\colon \mathfrak{g}\times TQ \rightarrow G\times (Q\times Q)$ on the Atiyah groupoid $(Q\times Q)\times G\rightrightarrows Q$. Then, $\widetilde{(\tau,{\mathcal R})}=(\widetilde{\tau},\widetilde{\mathcal R})\colon T^*\mathfrak{g}^*\times T^*T^*Q\rightarrow T^*G\times T^*(Q\times Q)$ is the cotangent lift of $(\tau,{\mathcal R})$ (see Section~\ref{Sec:PrincipalBundle}).

\end{example}

\section{Derivation of geometric integrators for reduced systems} \label{Sec:Integrator}

Given a symplectic groupoid  $G\rightrightarrows M$ (see Appendix~\ref{App:SymplGroupoid}) we know that there exists a unique Poisson structure $\Pi$ on $M$ for which $\beta\ \colon\  G\to M$ is a Poisson map, and $\alpha\ \colon\  G\to M$ is an anti-Poisson map (that is, $\alpha$ is a Poisson map when $M$ is equipped with the Poisson structure $-\Pi$).
In particular when we are dealing with the symplectic groupoid $T^*G$, where
  $G$ is a Lie groupoid, the Poisson structure on $A^*G$ is the
  (linear) Poisson structure of the dual of a Lie algebroid (see Remark~\ref{remark:AppB} in Appendix~\ref{App:SymplGroupoid}).

On the other hand, recall that a submanifold $N$ of a Poisson manifold $(M,\Pi)$ is said to be {\bf coisotropic} if the bracket of two smooth functions, defined on an open subset of $M$ and which vanish on $N$, vanishes on $N$ too. 
This allows to characterize a Poisson map since we know that the graph of $\varphi: M\rightarrow M$ is coisotropic in $M^-\times M$ if and only if $\varphi$ is a Poisson map \cite{weipoisson}. Here $M^-$ is the same manifold $M$ equipped with the Poisson structure $-\Pi$.

Our main objective is to obtain a new construction of geometric integrators while preserving the Lie-Poisson structure (a subject treated with different points of view, for instance,  in \cite{Zhong,channell,Modin,Ferraro-vaquero}). 
The next theorem is the core of our construction of Poisson integrators. In the statement of the theorem, we use bisections in a Lie groupoid that are defined in Appendix~\ref{App:lieGroupoid}.  
\begin{theorem}\label{PoissonIsomorphisms}[see \cite{coste}]\label{CDW} Let $G\stackrel[{\beta}]{{\alpha}}{\rightrightarrows} M$ be a symplectic groupoid. Let ${\mathcal L}
$ be a
  Lagrangian submanifold of $G$ such that $\alpha_{|{\mathcal L}}$ is a (local)
  diffeomorphism.  Then
\begin{enumerate}
\item $\beta_{|{\mathcal L}}:\lag\rightarrow M$ is a (local) diffeomorphism as well (so, ${\mathcal L}$ is a Lagrangian bisection).
\item The mapping
  $\hat{\mathcal L}=\alpha\circ(\beta_{|\mathcal L})^{-1}:M\rightarrow M$ and its
  inverse, $\beta\circ (\alpha_{|\mathcal L})^{-1}$,
  are (local) Poisson isomorphisms or, equivalently,
the image of ${\mathcal L}$ by the map 
 \[
\begin{array}{rrcl}
\alpha\times \beta:& G&\longrightarrow& M^-\times M\\
&g&\longmapsto&({\alpha}(g), {\beta}(g))
\end{array}
\] 
is a coisotropic submanifold of $M^-\times M$. 
\end{enumerate}
\end{theorem}

Theorem~\ref{PoissonIsomorphisms} contains the crucial properties of the cotangent Lie groupoid  $T^*G\stackrel[\tilde{\beta}]{\tilde{\alpha}}{\rightrightarrows} A^*G$ to derive Poisson integrators.

\subsection{Linear-Poisson integrators for Hamiltonian and Lagrangian systems}

Given a  Hamiltonian function  $H: A^*G\rightarrow {\mathbb R}$ (see \cite{weinstein96,LeMaMa}) 
we derive its associated  Lagrangian submanifold 
\[
dH(A^*G)=\hbox{Im}\, dH\subset T^*A^*G\; .
\]

If ${\mathcal R}: AG\rightarrow G$ is a generalized discretization map, then the corresponding cotangent lift
$T^*\mathcal R: T^*AG\rightarrow T^*G$ is a  symplectomorphism by construction 
and, from Theorem \ref{thm:CotangentLift}, $\widetilde{\mathcal R}= T^*{\mathcal R}\circ {\mathcal I}_{AG}\colon T^*(A^*G)\rightarrow T^*G$ is also a generalized discretization map. 

Therefore, from the definition of a discretization map, the zero section $$0_{T^*(A^*G)}\colon A^*G \rightarrow T^*(A^*G)$$ is sent to the image of the identity section $\tilde{\epsilon}\colon A^*G \rightarrow T^*G$ of the Lie groupoid $T^*G\rightrightarrows A^*G$. Moreover, the set $\tilde{\epsilon}(A^*G)$ is a horizontal Lagrangian submanifold of $(T^*G, \omega_G)$ (see Theorem B.6 in Appendix~\ref{App:SymplGroupoid}).

Denote by $h\, dH(A^*G)$ the multiplication by $h$ of the fiber elements with respect to the vector bundle structure $\pi_{A^*G}: T^*A^*G\rightarrow A^*G$. 
Locally, if 
$$dH(A^*G)=\left\{\left(q^i, y_{\gamma}; \frac{\partial H}{\partial q^i}, \frac{\partial H}{\partial y_{\gamma}}\right)\right\}\, ,$$ then 
$$h\, dH(A^*G)=\left\{\left(q^i, y_{\gamma}; h\frac{\partial H}{\partial q^i}, h\frac{\partial H}{\partial y_{\gamma}}\right)\right\}\; .$$

\begin{theorem}\label{Propo-Lagbi}
If $H: A^*G\rightarrow {\mathbb R}$ is a Hamiltonian function, then  $\widetilde{\mathcal R}(h\, dH(A^*G))$ is a Lagrangian bisection of the symplectic groupoid $(T^*G,\omega_G)\rightrightarrows A^*G$, for small enough $h$.  
\end{theorem}
\begin{proof}
For enough small $h$ 
\[
h\, dH(A^*G)\subset T^*A^*G\;
\]
is a horizontal Lagrangian submanifold of $(T^*A^*G, \omega_{A^*G})$ since it is close to the zero section (see Proposition 2.1 in \cite{Weinstein-annals}).  Therefore, $\widetilde{\mathcal R}(h\, dH(A^*G))$ is also a horizontal Lagrangian submanifold since it is close to  $\tilde{\epsilon}(A^*G)$. As a consequence $\widetilde{\mathcal R}(h\, dH(A^*G))$ is a Lagrangian bisection.
\end{proof}

Alternatively, let $L: AG\rightarrow {\mathbb R}$  be a regular Lagrangian function. This means that the Legendre transformation ${\mathcal F}L\colon AG \rightarrow A^*G$ given by 
$$\langle {\mathcal F}L(a_q) ,b_q \rangle =\left.\dfrac{\rm d}{{\rm d}t}\right|_{t=0} L(a_q+tb_q), \quad \mbox{for }  a_q,\;b_q\in A_qG, $$
is a local diffeomorphism. So, we can consider the Lagrangian energy $E_L\colon AG\rightarrow \mathbb{R}$ defined by 
$$E_L=\Delta(L)-L,$$
where $\Delta$ is the Liouville vector field on $AG$, and the Hamiltonian function $H\colon A^*G \rightarrow \mathbb{R}$ given by $$H=E_L\circ {\mathcal F}L^{-1},$$
see~\cite{LeMaMa}. Then, using Theorem \ref{Propo-Lagbi}, we have that $\widetilde{\mathcal R}(h\, dH(A^*G))$ is a Lagrangian bisection of $T^*G$ for $h$ small enough. 

On the other hand, one may prove that the following diagram
\begin{equation*}
    \xymatrix{ T^*(AG) \ar[rr]^{T^*{\mathcal R}}  && T^*G && T^*(A^*G)\ar[ll]_{\widetilde{\mathcal R} }\ar@/_2pc/[llll]^{{\mathcal I}_{AG}} \\ & AG \ar[ul]^{dL} \ar[rd]_{E_L} \ar[rr]^{{\mathcal F}L} &&A^*G  \ar[ur]_{dH} \ar[ld]^{H}& \\ && \mathbb{R} && }
\end{equation*} 
is commutative and 
\begin{equation}\label{eq:LPintegratordHdL}
\widetilde{\mathcal R}(dH({\mathcal F}L(a_q))=T^*{\mathcal R}(dL(a_q))\, , \; \forall a_q\in A_qG,\;  q\in Q\; .
\end{equation}
see~\cite{2006Grabowska}.

Now, denote by $h\, dL(AG)$ the multiplication by $h$ of the fiber elements with respect to the vector bundle structure $\zeta_{A^*G}: T^*AG\rightarrow A^*G$. 
Locally, if 
$$dL(AG)=\left\{\left(q^i, y^{\gamma}; \frac{\partial L}{\partial q^i}, \frac{\partial L}{\partial y^{\gamma}} \right)\right\}\, ,$$ then 
$$h\, dL(AG)=\left\{\left(q^i, hy^{\gamma}; h\frac{\partial L}{\partial q^i}, \frac{\partial L}{\partial y^{\gamma}}\right)\right\}\; .$$

From~\eqref{eq:LPintegratordHdL}, it follows that 
$$T^*{\mathcal R}(h\, dL(a_q))=\widetilde{\mathcal R}(h\,dH({\mathcal F}L(a_q))), \quad \forall \; a_q\in A_qG.$$
In particular, 
$$T^*{\mathcal R}(h\, dL(AG))=\widetilde{\mathcal R}(h\,dH(A^*G))\, .$$
Therefore, using Theorem~\ref{Propo-Lagbi} we obtain the following result.
\begin{theorem}\label{Propo-Lagbi2}
If $L: AG\rightarrow {\mathbb R}$ is a regular Lagrangian, then  $T^*{\mathcal R}(h\, dL(AG))$ is a Lagrangian bisection for $h$ enough small.
\end{theorem}

Using Theorem~\ref{PoissonIsomorphisms}, \ref{Propo-Lagbi} and \ref{Propo-Lagbi2} we derive the following construction of Poisson
geometric integrators via discretization maps.
\begin{table}[h!]{{\bf Integrators for  a Hamiltonian system on $A^*G$ 
}}
  \centering
  \begingroup
  \ttfamily\bfseries
  \renewcommand{\arraystretch}{1.2}
  \begin{tabular}{|l|}
  \hline
     input: ${H}: A^*G\rightarrow {\mathbb R}$, $\mu_0\in A^*G$ and $h>0$ \\
     for \(k = 1\) to \(N-1\)\\
    \ \ solve:  $\gamma_k=\widetilde{\alpha}^{-1}(\mu_{k-1})\cap \widetilde{\mathcal R}(h\, dH(A^*G))\subset T^*G$ \\
     \ \ define: $\mu_{k+1}=\tilde{\beta}(\gamma_k)$\\
     end for\\
     output: \(\mu_{N}\in A^*G\)\\
     \hline
  \end{tabular}
\endgroup
\end{table}
\begin{table}[h!]{{\bf Integrators for  a Lagrangian system on $AG$ 
}}
  \centering
  \begingroup
  \ttfamily\bfseries
  \renewcommand{\arraystretch}{1.2}
  \begin{tabular}{|l|}
  \hline
     input: ${L}: AG\rightarrow {\mathbb R}$, $v_0\in AG$ and $h>0$\ \ \ \ $\,$\\
     define $\mu_0={\mathcal F}L (v_0)\in A^*G $\\
     for \(k = 1\) to \(N\)\\
\ \    solve:  $\gamma_k=\widetilde{\alpha}^{-1}(\mu_{k-1})\cap T^*{\mathcal R}(h\, dL(AG))\subset T^*G$ \\
 \ \    define: $\mu_{k}=\tilde{\beta}(\gamma_k)$\\
     end for\\
      solve: $v_N$ such that $\mu_N={\mathcal F}L (v_N)$\\
     output: \(v_{N}\in AG\)\\
     \hline
  \end{tabular}
\endgroup
\end{table}

\subsection{Lie-Poisson integrators on Lie groups}

Let $G$ be a Lie group and ${\mathfrak g}$ the associated Lie algebra. Consider a retraction map $\tau: {\mathfrak g}\rightarrow G$ that is also a discretization map on the Lie group. 
For a regular Lagrangian function $L: {\mathfrak g}\rightarrow {\mathbb R}$ the geometric integrator preserving the Lie-Poisson structure on ${\mathfrak g}^*$ is
\begin{align*}
\mu_k&=(T_e (\tau^{-1}\circ r_{\tau(h\xi)}))^*\frac{\partial L}{\partial \xi}(\xi)\, ,\\
\mu_{k+1}&=(T_e (\tau^{-1}\circ l_{\tau(h\xi)}))^*\frac{\partial L}{\partial \xi}(\xi)\, ,
\end{align*}
or equivalently,
\begin{align*}
\mu_k&=(d^R\tau_{h\xi}^{-1})^*\frac{\partial L}{\partial \xi}(\xi) \, ,\\
\mu_{k+1}&=(d^L\tau_{h\xi}^{-1})^*\frac{\partial L}{\partial \xi}(\xi)=Ad^*_{\tau(h\xi)}\mu_k  \, ,
\end{align*}
using right and left trivializations \cite{bourabee04} where
\[
T_{\xi}\tau(\eta)= T_e r_{\tau(\xi)} (d^R\tau_{\xi}(\tau))
= T_e l_{\tau(\xi)} (d^L\tau_{\xi}(\tau))\, .
\]
We obtain the corresponding geometric integrator in ${\mathfrak g}$ defining $v_k=({\mathcal F}L)^{-1} (\mu_k), k=0, \ldots, N$.

In the Hamiltonian framework for $H: {\mathfrak g}^*\rightarrow {\mathbb R}$,  we obtain
\begin{align*}
\mu_k&=(d^R\tau_{h\partial H/\partial \mu}^{-1})^*\mu  \, ,\\
\mu_{k+1}&=Ad^*_{\tau(h{\partial H/\partial \mu})}\mu_k  \, .
\end{align*}

\begin{example} \textbf{(Rigid body)}\label{example-rigidbody}
 The configuration space for the rigid body  is the Lie group $G=SO(3)$. As usual,  we identify the Lie algebra $\mathfrak{so}(3)$ with $\mathbb{R}^3$ using the isomorphism $\widehat{\cdot}\ \colon\ \mathbb{R}^3 \to \mathfrak{so}(3)$ given by
\[
\widehat{\omega}=\begin{pmatrix}
  0&-z&y\\
z&0&-x\\
-y&x&0
\end{pmatrix} \, .
\]
The angular velocity in body coordinates is denoted by $\Omega$, and $\hat\Omega=R ^{-1}\dot R$.
The moment of inertia tensor in body coordinates is $\mathbb{I}=\operatorname{diag}(I_1,I_2,I_3)$.
Also, the body angular momentum is $\Pi=\mathbb{I}\Omega$, so in principal axes, $\Pi=(\Pi_1,\Pi_2,\Pi_3)=(I_1\Omega_1,I_2\Omega_2,I_3\Omega_3)$. The Hamiltonian in these variables is
\[
H= \frac{1}{2}\left( \frac{\Pi_1^2}{I_1}+\frac{\Pi_2^2}{I_2}+\frac{\Pi_3^2}{I_3}\right),
\]
which we regard as a function $H\colon \mathfrak{so}(3)^* \to
\mathbb{R}$. Note that if $SO(3)$ is considered as a Lie groupoid, then the dual Lie algebroid $A^*G$ is precisely $\mathfrak{so}^*(3)$. We use as a retraction map the Cayley map $\operatorname{cay} \colon\mathfrak{so}(3) \to SO(3)$,
\[
\operatorname{cay}(\hat\omega)=I_3+\frac{4}{4+\|\omega\|^{2}}\left(\hat\omega+\frac{\hat\omega^{2}}{2}\right),
\]
where $I_{3}$ is the $3\times 3$ identity matrix. We have that
\begin{align*}
d^L\operatorname{cay}^{-1}_{\omega}&=
\left(1+\frac{1}{4}\|\omega\|^{2}\right)I_{3}-\frac{1}{2}\hat{\omega}+\frac{1}{4}\hat{\omega}^2\, ,\\
d^R\operatorname{cay}^{-1}_{\omega}&=
\left(1+\frac{1}{4}\|\omega\|^{2}\right)I_{3}+\frac{1}{2}\hat{\omega}+\frac{1}{4}\hat{\omega}^2\, ,
\end{align*}
where $\hat{\omega}\in \mathfrak{so}(3)$.
The corresponding equations for the Hamiltonian $H$ are
\[
\left(\begin{array}{c}
(\mu_k)_1\\
(\mu_k)_2\\
(\mu_k)_3
\end{array}
\right)
= \left( \begin {array}{c}  \left( \frac{x^2}{4}+1 \right) {{\Pi}_1}+
 \left(\frac{xy}{4}+ \frac{z}{2} \right) { {{\Pi}_2}}+ \left( \frac{xz}{4}-\frac{y}{2}
 \right) {{\Pi}_3}\\ \left( \frac{xy}{4}-\frac{z}{2} \right) 
 {{{\Pi}_1}}+ \left( \frac{y^2}{4}+1 \right) {{{\Pi}_2}}+ \left( \frac{yz}{4}+\frac{x}{2}
 \right)  {{\Pi}_3}\\ \left( \frac{xz}{4}+\frac{y}{2} \right) 
{{{\Pi}_1}}+ \left( \frac{yz}{4}-\frac{x}{2}
 \right) { {{\Pi}_2}}+ \left( \frac{z^2}{4}+1 \right) { {{\Pi}_3}}\end {array} \right),
\]
\[
\left(\begin{array}{c}
(\mu_{k+1})_1\\
(\mu_{k+1})_2\\
(\mu_{k+1})_3
\end{array}
\right)= \left( \begin {array}{c}  \left( \frac{{x}^{2}}{4}+1 \right) { {{\Pi}_1}}+
 \left( \frac{xy}{4}-\frac{z}{2} \right) {{\Pi}_2}+ \left( \frac{xz}{4}+\frac{y}{2}
 \right) { {{\Pi}_3}}\\ \left(\frac{xy}{4}+ \frac{z}{2} \right) {{\Pi}_1}+ \left( \frac{y^2}{4}+1 \right) {{\Pi}_2}+ \left( \frac{yz}{4}-\frac{x}{2} \right) {{\Pi}_3}\\ \left( \frac{xz}{4}-\frac{y}{2}
 \right) {{\Pi}_1}+ \left( \frac{yz}{4}+ \frac{x}{2} \right) {{\Pi}_2}+ \left( \frac{z^2}{4} +1\right) {{\Pi}_3}\end {array} \right),
\]
where $x=\tfrac{h{\Pi}_1}{I_1}$, $y=\tfrac{h{\Pi}_2}{I_2}$ and $z=\tfrac{h{\Pi}_3}{I_3}$. Thus, the Lie-Poisson integrator $
\varphi_h: {\mathfrak g}^*\rightarrow {\mathfrak g}^*$ is given by: 
\[
\varphi_h \left((\mu_k)_1,
(\mu_k)_2,
(\mu_k)_3)\right)=\left((\mu_{k+1})_1,
(\mu_{k+1})_2,
(\mu_{k+1})_3\right)\; .
\]
\end{example}
\subsection{Systems defined on a trivial principal bundle}\label{Sec:PrincipalBundle}

Let $G$ be a Lie group with Lie algebra ${\mathfrak g}$ and $Q$ a smooth manifold. Consider  the trivial left action of $G$ on the product manifold $G \times Q$, that is, $\Psi_g\colon G\times Q \rightarrow G\times Q$,
\[
\Psi_g(\bar{g}, q)=(g\bar{g}, q)\, .
\]
The space of orbits of the tangent lift of $\Psi_g$ of $G$ on $T(G \times Q)$ is just the Atiyah algebroid 
\[
T(G \times Q)/G \simeq {\mathfrak g} \times TQ \rightarrow  Q.
\]
In this particular case, a discretization map for this algebroid is given by a map from $
 {\mathfrak g} \times TQ$ to $
G\times (Q\times Q)$.
We need in this case to consider a discretization map on $G$ given by $\tau: {\mathfrak g}\rightarrow G$ and a discretization map ${\mathcal R}: TQ\rightarrow Q\times Q$ to obtain $\widetilde{\mathcal R}= (\tau, {\mathcal R})$, as in Example~\ref{Ex:Cotangent}.

For a Lagrangian function $L: {\mathfrak g} \times TQ\rightarrow {\mathbb R}$
the discrete equations of motion are:
\begin{align*}
\mu_k&=(d^R\tau_{h\xi}^{-1})^*\frac{\partial L}{\partial \xi}(\xi, {\mathcal R}^{-1}(q_k, q_{k+1}))\, ,\\
\mu_{k+1}&=(d^L\tau_{h\xi}^{-1})^*\frac{\partial L}{\partial \xi}(\xi, {\mathcal R}^{-1}(q_k, q_{k+1} ))=Ad^*_{\tau(h\xi)}\mu_k\, ,\\
&(q_k, -p_k;  q_{k+1}, p_{k+1})\in T^* {\mathcal R}(h \, dL_{\xi}(TQ))\, ,
\end{align*}
where $L_{\xi}: TQ\rightarrow {\mathbb R}$ is defined by 
$L_{\xi}(v_q)=L(\xi, v_q)$ for all $v_q\in T_qQ$.
The previous methods correspond to Poisson integrators for the  Lagrange-Poincar\'e equations (see \cite{CMR01a}).

\subsection{Poisson integrators for action Lie algebroids}\label{section7.4}

Another relevant case for mechanical systems with symmetry is when the Lie groupoid is the action Lie groupoid (see Example \ref{example-action}).
Given a discretization map $\tau: {\mathfrak h}\rightarrow H$ it is possible to check that the map
${\mathcal R}\colon Q \rtimes \mathfrak{h}\rightarrow Q \rtimes H$ defined by
${\mathcal R}(q, \xi)=(q, \tau(\xi))$ is a discretization map for the action Lie groupoid. 
Therefore, a regular Lagrangian function $L: Q \rtimes \mathfrak{h}\rightarrow {\mathbb R}$ defines
\[
dL(Q \rtimes \mathfrak{h})\subset T^*( Q \rtimes \mathfrak{h})
\]
as follows
\[
dL(Q \rtimes \mathfrak{h})=\left\{\left(q, \xi; \frac{\partial L}{\partial q}, \frac{\partial L}{\partial \xi}\right)\right\}\, .
\]
Thus,
\[
h\, dL(Q \rtimes \mathfrak{h})=\left\{\left(q, h\xi; h\frac{\partial L}{\partial q}, \frac{\partial L}{\partial \xi}\right)\right\}\, .
\]
On the other hand, we have the discretization map
\[
T^*{\mathcal R}(q, \xi; p_q, p_{\xi})=(q, \tau(\xi); p_q, (T_{\tau(\xi)}\tau^{-1})^* p_{\xi}) \,.
\]
As a result, 
\[
T^*{\mathcal R}(h\, dL(q,\xi))=\left(q, \tau(h\xi); h\frac{\partial L}{\partial q}, (T_{\tau(h\xi)}\tau^{-1})^* \frac{\partial L}{\partial \xi}\right) \, .
\]
Now, using the expressions of the structural maps of the action Lie groupoid in Appendix~\ref{App:action}, it follows that the right and left translations in this Lie groupoid are given by 
$$R_{(q,g)}(q',g')=(q',R_g(g'))=(q',g'g),$$
for $(q',g')\in \beta^{-1}(\alpha(q,g))$ and 
$$L_{(q,g)}(q',g')=(q,L_g(g'))=(q,gg'),$$
for $(q',g')\in \alpha^{-1}(\beta(q,g))$. Equations~\eqref{eq:ApB_source} and~\eqref{eq:ApB_target} in Appendix~\ref{App:SymplGroupoid} and the definition of the inversion map in the action Lie groupoid $Q\rtimes  H\rightrightarrows Q$ in Appendix~\ref{App:action} leads to the following source and target maps for the symplectic groupoid $T^*(Q \rtimes \mathfrak{h})\rightrightarrows Q\times \mathfrak{h}^*$:
\begin{align*}
\widetilde{\alpha}(q, g, p_q, p_g)&=(q, -J(p_q)+(T_eR_g)^*p_g)\, ,\\
\widetilde{\beta}(q, g, p_q, p_g)&=(qg, (T_eL_g)^*p_g)\, ,
\end{align*}
where $\langle J(p_q), \eta\rangle=\langle p_q, \eta_Q(q)\rangle$ is the momentum map and $\eta_Q$ is the infinitesimal generator corresponding  to $\eta\in {\mathfrak h}$.

Therefore applying our algorithm we derive the following geometric integrator:
\begin{align*}
\mu_k&=-J\left(h\, \frac{\partial L}{\partial q}(q_k, \xi)\right)+(d^R\tau^{-1}_{h\xi})^*\frac{\partial L}{\partial \xi}(q_k, \xi)\, ,\\
q_{k+1}&=q_k \tau(h\xi)\, ,\\
\mu_{k+1}&=(d^L\tau^{-1}_{h\xi})^*\frac{\partial L}{\partial \xi}(q_k, \xi)\, .
\end{align*}
In the Hamiltonian case, for a Hamiltonian function
$H:Q\times \mathfrak{h}^*\rightarrow {\mathbb R} $
the geometric integrator is
\begin{align}
\mu_k&=J\left( h\frac{\partial H}{\partial q}(q_k, \mu)\right)+(d^R\tau^{-1}_{h\frac{\partial H}{\partial \mu}})^*\mu\, ,\nonumber\\
q_{k+1}&=q_k \tau\left(h\frac{\partial H}{\partial \mu}\right)\, ,\label{equation-action}\\
\mu_{k+1}&=(d^L\tau^{-1}_{h\frac{\partial H}{\partial \mu}})^*\mu\, .\nonumber
\end{align}
From the first equation, given $(q_k, \mu_k)\in Q\times \mathfrak{h}^*$, we obtain $\mu\in \mathfrak{h}^*$. From the second and third equation we derive the value of $(q_{k+1}, \mu_{k+1})$ in $Q\times \mathfrak{h}^*$. 

The Lie algebroid structure on $Q\rtimes \mathfrak{h}\rightarrow Q$ induces a linear Poisson structure on the dual bundle $Q\times \mathfrak{h}^*\rightarrow Q$ (see Appendix~\ref{App:lieGroupoid}) given by 
\begin{align*}
\{q^i, q^j\}&=0\, , \\
\{\mu_i, q^j\}&=\rho^j_i\, ,\\
\{\mu_i, \mu_j\}&={\mathcal C}_{ij}^k \mu_k\, ,
\end{align*}
where  $(q^i, \mu_j)$ are coordinates on $Q\times \mathfrak{h}^*$ defined as follows. The coordinates on $Q$ are $(q^i)$. Now, if $\{e_i\}$ is a basis of the Lie algebra $\mathfrak{h}$, it induces a dual basis $\{e^i\}$ on $\mathfrak{h}^*$ with associated coordinates $(\mu_i)$. Thus, the coefficients on the Poisson bracket
are defined as follows
$
[e_i, e_j]={\mathcal C}_{ij}^ke_k
$
and 
$(e_i)_Q(q)=\rho^j_i(q)\frac{\partial}{\partial q^j}$.

The key point is that the previous geometric integrator (\ref{equation-action}) preserves by construction the Poisson bracket $\{\cdot,\cdot\}$ on $Q\times \mathfrak{h}^*$.
\begin{example}{(Heavy top)}\label{example:heavy}
A typical and interesting example defined on an action Lie algebroid is the heavy top. The system is defined by a Lagrangian $L: S^2\rtimes \mathfrak{so}(3) \rightarrow {\mathbb R}$ given by
\[
L(q, \xi)=\frac{1}{2}\xi\cdot {I}\xi-(mgd) q\cdot {\rm e}\, ,
\]
where $\xi \in \mathfrak{so}(3)$ represents the angular velocity, $q\in S^2$ is the direction opposite to the gravity and ${\rm e}=({\rm e_1}, {\rm e}_2, {\rm e}_3)$ is a unit vector on ${\mathbb R}^3$ in the direction from a fixed point to the center of mass of the heavy top. Moreover, $m$ represents the mass of the body, $g$ the gravitational acceleration and $d$ the distance from the fixed point to the center of mass. The inertia tensor $I$ of the body is assumed to be a diagonal matrix with diagonal elements $(I_1, I_2, I_3)$. The action of the special orthogonal group $SO(3)$ on $S^2$ is the standard one.

As in Example~\ref{example-rigidbody} we use a discretization map based on the Cayley map
\[
{\mathcal R}(q, \xi)=(q, {\rm cay}(\xi))\, .
\]
Therefore, the corresponding Poisson integrator is 
\begin{align*}
&\left(\begin{array}{c}
(\mu_k)_1\\
(\mu_k)_2\\
(\mu_k)_3
\end{array}
\right)
=
\left(
\begin{array}{c}
mgd\,((q_k)_2{\rm e}_3-(q_k)_3{\rm e}_2)\\
mgd\,((q_k)_3{\rm e}_1-(q_k)_1{\rm e}_3)\\
mgd\, ((q_k)_3{\rm e}_2-(q_k)_2{\rm e}_3)
\end{array}\right)
\\&+
\left( \begin {array}{c}  I_1\left( \frac{(h\xi_1)^2}{4}+1 \right) {{\xi}_1}+
 I_2\left(\frac{h^2\xi_1\xi_2}{4}+ \frac{h\xi_3}{2} \right) { {{\xi}_2}}+ I_3\left( \frac{h^2\xi_1\xi_3}{4}-\frac{h\xi_2}{2}
 \right) {{\xi}_3}\\ I_1\left( \frac{h^2\xi_1\xi_2}{4}-\frac{h\xi_3}{2} \right) 
 {{{\xi}_1}}+ I_2\left( \frac{(h\xi_2)^2}{4}+1 \right) {{{\xi}_2}}+ I_3\left( \frac{h^2\xi_2\xi_3}{4}+\frac{h\xi_1}{2}
 \right)  {{\xi}_3}\\ I_1\left( \frac{h^2\xi_1\xi_3}{4}+\frac{h\xi_2}{2} \right) 
{{{\xi}_1}}+ I_2\left( \frac{h^2\xi_2\xi_3}{4}-\frac{h\xi_1}{2}
 \right) { {{\xi}_2}}+ I_3\left( \frac{(h\xi_3)^2}{4}+1 \right) { {{\xi}_3}}\end {array} \right),\\
 &q_{k+1}=q_k{\rm cay}(h\xi)\, ,\\
&\left(\begin{array}{c}
(\mu_{k+1})_1\\
(\mu_{k+1})_2\\
(\mu_{k+1})_3
\end{array}
\right)
\\&=\left( \begin {array}{c}  I_1\left( \frac{(h\xi_1)^2}{4}+1 \right) {{\xi}_1}+
 I_2\left(\frac{h^2\xi_1\xi_2}{4}- \frac{h\xi_3}{2} \right) { {{\xi}_2}}+ I_3\left( \frac{h^2\xi_1\xi_3}{4}+\frac{h\xi_2}{2}
 \right) {{\xi}_3}\\ I_1\left( \frac{h^2\xi_1\xi_2}{4}+\frac{h\xi_3}{2} \right) 
 {{{\xi}_1}}+ I_2\left( \frac{(h\xi_2)^2}{4}+1 \right) {{{\xi}_2}}+ I_3\left( \frac{h^2\xi_2\xi_3}{4}-\frac{h\xi_1}{2}
 \right)  {{\xi}_3}\\ I_1\left( \frac{h^2\xi_1\xi_3}{4}-\frac{h\xi_2}{2} \right) 
{{{\xi}_1}}+ I_2\left( \frac{h^2\xi_2\xi_3}{4}+\frac{h\xi_1}{2}
 \right) { {{\xi}_2}}+ I_3\left( \frac{(h\xi_3)^2}{4}+1 \right) { {{\xi}_3}}\end {array} \right) \, ,
\end{align*}
where $q_k=((q_k)_1, (q_k)_2, (q_k)_3)\in S^2$ and
\[
\widehat{\xi}=\left(
\begin{array}{ccc}
0&-\xi_3&\xi_2\\
-\xi_2&0&-\xi_1\\
-\xi_2&\xi_1&0
\end{array}
\right)\in \mathfrak{so}(3)\, .
\]
    
\end{example}

    \subsection{Higher-order Poisson integrators  using composition of Lagrangian submanifolds }\label{Sec:Composition}
The idea of composing methods is important to derive higher order methods \cite{hairer}. In the cases presented in this paper, the composition is  interpreted as the composition of Lagrangian submanifolds of the cotangent groupoid.  

If ${\mathcal L}_1$ and ${\mathcal L}_2$ are Lagrangian submanifolds of the cotangent groupoid $(T^*G, \omega_G)$, then the composition 
\[
{\mathcal L}_2\circ {\mathcal L}_1\subset T^*G
\]
is a Lagrangian submanifold (under clean intersection conditions, see \cite{guilleminsternberg80})
where 
\[
{\mathcal L}_2\circ {\mathcal L}_1=\{ \gamma^{12}\in T^*G\; |\; \exists\;  \gamma^1\in {\mathcal L}_1,\gamma^2\in {\mathcal L}_2,  \mbox{  such that } \gamma^{12}=\tilde{m}(\gamma^1, \gamma^2)\}\, .
\]
Therefore, from a Lagrangian function $L: AG\rightarrow {\mathbb R}$ and a discretization map ${\mathcal R}: AG\rightarrow G$ we obtain the Lagrangian submanifold ${\mathcal L}^h=T^* {\mathcal R}(h\, dL(AG))$ of $T^*G$.  The composition with step sizes $a_1h, \ldots, a_sh$ is the Lagrangian submanifold
\begin{equation}\label{eq: Lagcomp}
{\mathcal L}^h_{a_1,\ldots, a_s}={\mathcal L}^{a_sh}\circ \ldots \circ {\mathcal L}^{a_1h}\, ,
\end{equation}
that generates a Poisson  composition method:
\begin{align*} 
\gamma&=\widetilde{\alpha}^{-1}(\mu_{k})\cap T^* {\mathcal R}({\mathcal L}^h_{a_1,\ldots, a_s})\subset T^*G\, , \\
    \mu_{k+1}&=\tilde{\beta}(\gamma)\, .
\end{align*}  
If ${\mathcal R}$ is a symmetric discretization map and the coefficients $a_1, \ldots, a_s$ verifies 
\begin{align*}
a_1+\ldots+a_s&=1\, ,\\
a^3_1+\ldots+a^3_s&=0\, ,
\end{align*}
then the Poisson method derived from ${\mathcal L}^h_{a_1,\ldots, a_s}$ is at least of order 3 (see \cite{hairer}). 

Another interesting possibility consists on  using  the discretization map ${\mathcal R}$ and the corresponding adjoint  map ${\mathcal R}^*$. If we denote by
${\mathcal L}^h_*=T^*{\mathcal R}^*(h\,dL(AG))$,  then the operation
\begin{equation}\label{eq:comp-3}
{\mathcal L}_*^{b_sh}\circ {\mathcal L}^{a_sh}\circ \ldots \circ {\mathcal L}^{b_1h}_*\circ {\mathcal L}^{a_1h}
\end{equation}
generates a Poisson composition method for appropriate conditions on the coefficients $a_l$ and $b_l$, $1\leq l\leq s$, (see \cite{hairer}).
For instance, 
\begin{equation}\label{eq:comp-3-example}
 {\mathcal L}^{h/2}\circ {\mathcal L}^{h/2}_*
\end{equation}
is always a second-order method.

\section{Conclusions and future work}\label{Sec:conclusion}
In this paper we have introduced the notion of retraction and discretization maps for Lie groupoids showing the potential to obtain Poisson integrators for different types of Hamiltonian and Lagrangian systems. Some interesting future developments are briefly commented in the following points. 

\begin{enumerate}
\item Given a discretization map ${\mathcal R}: AG\rightarrow G$ and a Lagrangian function $L: AG\rightarrow {\mathbb R}$ we can define the following discrete Lagrangian $L_d: G\rightarrow {\mathbb R}$
\[
L_d(g)=h L\left( \frac{1}{h}{\mathcal R}^{-1}(g)\right)\; .\]
Using this type of  discrete Lagrangians it is possible to derive new methods as in Chapter 2 in \cite{marsden-west, MMM06Grupoides}.

\item Another different technique to derive geometric integrators is based on splitting techniques \cite{hairer}. Assume that the initial Hamiltonian function $H: A^*G\rightarrow {\mathbb R}$ splits as
$H=H_1+H_2$ where $H_i: A^*G\rightarrow {\mathbb R}$. Assuming that it is possible to find the exact flows
\[
\varphi_h^{H_1}, \quad \varphi_h^{H_2}\, ,
\]
then it is possible to obtain Poisson integrators using, for instance,
\[
\varphi_h^{H_1}\circ \varphi_h^{H_2}, \quad\varphi_h^{H_2}\circ \varphi_h^{H_1},\quad \varphi_{h/2}^{H_1}\circ \varphi_{h}^{H_2}\circ \varphi_{h/2}^{H_1}\, \ldots
\]
These Poisson integrators numerically solve the dynamics of the original system $H: A^*G\rightarrow {\mathbb R}$.  In a future paper, we will compare the splitting methods with our methods based on retraction and discretization maps. 
\item The proposed techniques in this paper are also adaptable to other interesting cases: systems with holonomic and nonholonomic constraints, systems with external forces, explicit time-dependent systems... 

\item Lagrangian mechanics on Lie algebroids (see \cite{weinstein96}) gives a unified point of view beyond the classical systems defined on the tangent bundle of the configuration manifold, that includes systems determined by Lagrangian functions on Lie algebras, Atiyah algebroids, action Lie algebroids (see Section \ref{Sec:Integrator}). The evolution of all these systems  are jointly described by the  integral curves of a second order differential equation defined on the corresponding Lie algebroid. In a future paper we will develop geometric integrators for these second order differential equations. As a consequence, we will obtain geometric integrators for the equations of systems 
reduced by a symmetry Lie group and, therefore,  automatically
we will preserve the associated momentum map \cite{Marsden-Ratiu}.
For instance,  consider a first-order differential equation on a Lie algebra: $\dot{\xi}=\Gamma(\xi)$  derived from the forced Euler-Poincar\'e equations: 
\[
\frac{d}{dt}\left(\frac{\delta l}{\delta \xi}\right)=ad^*_{\xi} \left(\frac{\delta l}{\delta \xi}\right)+F(\xi)\, 
\]
where $l: {\mathfrak g}\rightarrow {\mathbb R}$ is a Lagrangian defined on a Lie algebra and $F: {\mathfrak g}\rightarrow {\mathfrak g}^*$ is an external force (see, for instance, \cite{double}). Then, from a discretization map
$R_d: TG\rightarrow G\times G$ given by
${ R}_d(g, \dot{g})=(g, g{\rm exp}(g^{-1}\dot{g}))$, we obtain the generalized discretization map using the tangent lift $TR_d: TTG\rightarrow T(G\times G)$(see Section \ref{Sec:Lift}) and then since it is a $G$-equivariant discretization map (see Section \ref{Sec:reduction}) obtain
a new generalized discretization map ${\mathcal R}\equiv\widehat{TR_d}: TTG/G\rightarrow T(G\times G)$ that after left-trivialization we obtain: 
\[
\begin{array}{rcl}
{\mathcal R}:TTG/G\equiv {\mathfrak g}\times {\mathfrak g}\times {\mathfrak g}&\rightarrow& T(G\times G)/G\equiv {\mathfrak g}\times { G}\times {\mathfrak g}\\
(\xi_1, \xi_2, \xi_3)&\longmapsto&(\xi_2, {\rm exp}(\xi_1), Ad_{{\rm exp}(\xi_1)^{-1}}\xi_2+d^L{\rm exp}_{\xi_1}(\xi_3))\, .
\end{array}
\]
The first-order differential equation $\dot{\xi}=\Gamma(\xi)$ can be interpreted as the  section 
$\xi\longmapsto (\xi, \xi, \Gamma(\xi))$
of $\hbox{pr}_1: {\mathfrak g}\times {\mathfrak g}\times {\mathfrak g}\rightarrow {\mathfrak g}$ (see \cite{MMM3-1}).
Considering now the set
\[
\Gamma_d=\{{\mathcal R}(h\xi, \xi, h \Gamma(\xi)), \xi\in {\mathfrak g}\}\subset {\mathfrak g}\times { G}\times {\mathfrak g}
\]
we obtain  that the  corresponding discrete equation 
\begin{align*}
\frac{\xi_{k+1}-\xi_k}{h}&=d^L{\rm exp}_{h\xi_k}(\Gamma(\xi_k))\\
&=d^L{\rm exp}_{h\xi_k}\left(W_k^{-1}
ad^*_{\xi_k} \left(\frac{\delta l}{\delta \xi}(\xi_k)\right)+W_k^{-1}F(\xi_k)\right)\; ,
\end{align*}
where $W_k$ is the hessian matrix of $l$: $W_k=\left(\frac{\partial ^2 l}{\partial \xi^2}(\xi_k)\right)$.
To obtain the dynamics on $TG$ we need to add the reconstruction equation: $g_{k+1}=g_k\,{\rm exp}(h\xi_k)$.
These methods can be used to derive, for instance, Newmark methods for this extended class of second-order systems (see \cite{Newmark,Kane}). 

\item  Applications to optimization in neural networks. 
A promising application of the notion of retraction and discretization maps on Lie groupoids is to derive  novel techniques in the design of neural networks and their relation to obtaining minima for loss functions with symmetry (see \cite{10.5555/3540261.3542225,Zhao2022SymmetriesFM} and references therein). 

In neural networks the parameters are weights $W_a\in {\mathbb R}^{n_a\times n_{a-1}}$ for each layer $1\leq a\leq f$, the total parameter space is 
${\bf Param}={\mathbb R}^{n_{f}\times n_{{f-1}}}\times \ldots \times{\mathbb R}^{n_1\times n_0}$ and the loss function
$
{\mathcal L}: {\bf Param}={\mathbb R}^{N}\rightarrow {\mathbb R}
$
defined in terms of a  Data given as a set of  ${\mathbb R}^{n_0\times n_{f}}$ and $N=n_0(n_1\ldots n_{f-1})^2n_{f}$. For simplicity, we assume  that the parameters are defined on an euclidean space but it is possible to generalize it to a differentiable manifold structure. 

In many cases of interest the function ${\mathcal L}$ has 
some symmetries, whose study is a key design principle for neural networks. 
In particular, in many cases of interest, the symmetry is represented by a Lie group of symmetries $G$ \cite{AM87} and it is possible to apply standard reduction techniques  and to obtain a reduced loss function $\tilde{\mathcal L}: {\bf Param}/G\rightarrow {\mathbb R}$. 
This is for instance the case of  neural networks with normalization layers which introduce a scale symmetry in the Loss function ${\mathcal L}$ (as in the case of the popular batch normalization).
 The techniques developed in this paper allow to analyze in an intrinsic way the important problem of symmetry breaking \cite{10.5555/3540261.3542225} and the importance in neural network problems in the reduced setting working with Bregmann Lagrangians defined on $T{\bf Param}/G$ which is a Lie algebroid.  Using a generalized discretization map  we can construct a discrete Lagrangian defined on the Lie groupoid $({\bf Param}\times {\bf Param})/G$ and  we can directly apply  discrete variational techniques, as in \cite{campos}, using  the framework  developed by some of the authors in \cite{MMM06Grupoides}.

\end{enumerate}

 \begin{appendices}
 	\section{Lie groupoids}\label{App:lieGroupoid}

    We review some definitions and basic constructions on Lie groupoids and algebroids. For more on the subject, we refer the reader to the monographs	~\cite{2021CrainicBook} and K.\ Mackenzie's
 	book \cite{Mackenzie}.
    
 	As associated to every Lie groupoid there is a Lie algebroid, we first recall the notion of a Lie algebroid.
  
 	A \textbf{Lie algebroid} is a vector bundle $\tau:A\rightarrow M$
 	endowed with the following data:
 	\begin{itemize}
 		\item A bundle map $\rho: A\rightarrow TM$ called the anchor map.
 		
 		\item A Lie bracket $\lcf\cdot, \cdot \rcf$  on the space of sections $\Gamma(\tau)$
 		satisfying the Leibniz identity, i.e.
 		\[
 		\lcf X, fY\rcf=f\lcf X, Y\rcf+\rho(X)(f)Y
 		\]
 		for all $X, \ Y\in \Gamma(\tau)$ and any $f\in C^\infty(M)$.
 	\end{itemize}

  Thus a Lie algebroid is defined by the triple $(\tau\colon A\rightarrow M, \rho, \lcf \cdot,\cdot \rcf)$.
  
  We emphasize that every Lie algebroid structure $(\lcf\cdot,\cdot \rcf, \rho)$ on a vector bundle $\tau\colon A\rightarrow M$ induces a linear Poisson structure $\Lambda_{A^*}$ on the dual bundle $\tau^*: A^* \rightarrow M$. In fact, the Poisson bracket $\{\cdot, \cdot\}$ on $A^*$ is characterized by the following conditions:
  $$\{X^l,Y^l\}=-\lcf X,Y\rcf^l, \quad \{X^l,g\circ \tau^*\}=-\rho(X)(g)\circ \tau^*, \quad \{f\circ \tau^*,g\circ \tau^*\}=0, $$
 	for $X,\, Y \in \Gamma(A)$, $f,\,g\in {\mathcal C}^\infty(M)$ and $X^l\colon A^*\rightarrow \mathbb{R}$ is the fiberwise linear function on $A^*$ given by 
    $$X^l(\mu_x)=\langle \mu_x,X(x) \rangle, \quad \mbox{for } \mu_x\in A^* \mbox{ and 
    } x\in M.$$

 	 	\paragraph{Groupoids.}
 	A groupoid is a set $G$ equipped with the following  data:
 	\begin{enumerate}
 		\item Another set $M$, called the base;
 		\item Two surjective maps $\alpha\colon G\to M$ and $\beta\colon G\to M$, called, respectively,  the source and target projections. So, we visualize an element $g\in G$ as an arrow from $\alpha(g)$ to $\beta(g)$:
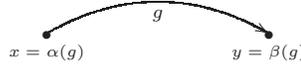
 		$$
 		\xymatrix{*=0{\stackrel{\bullet}{\mbox{\tiny
 						$x=\alpha(g)$}}}{\ar@/^1pc/@<1ex>[rrr]_g}&&&*=0{\stackrel{\bullet}{\mbox{\tiny
 						$y=\beta(g)$}}}}
 		$$
 		\item A partial multiplication, or composition map, $m\colon G_2\to G$ defined on the subset $G_2$ of $G\times G$:
 		\[
 		G_2=\left\{ (g,h)\in G\times G\mid \beta(g)=\alpha(h) \right\}.
 		\]
 		The multiplication is denoted by $m(g, h)=gh$ and it verifies the following properties:
 		\begin{enumerate}
 			\item $\alpha(gh)=\alpha(g)$ and $\beta(gh)=\beta(h)$.
 			\item $(gh)k=g(hk)$.
 			$$
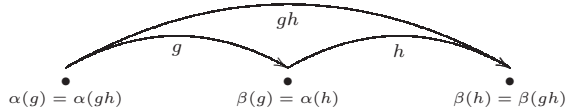
\xymatrix{*=0{\stackrel{\bullet}{\mbox{\tiny
 							$\alpha(g)=\alpha(gh)$}}}{\ar@/^2pc/@<2ex>[rrrrrr]_{gh}}{\ar@/^1pc/@<2ex>[rrr]_g}&&&*=0{\stackrel{\bullet}{\mbox{\tiny
 							$\beta(g)=\alpha(h)$}}}{\ar@/^1pc/@<2ex>[rrr]_h}&&&*=0{\stackrel{\bullet}{\mbox{\tiny
 							$\beta(h)=\beta(gh)$}}}}$$
 		\end{enumerate}
 		\item An identity section $\epsilon\colon  M \to G$ such that
 		\begin{enumerate}
 			\item $\epsilon(\alpha(g))g=g$ and $g\epsilon(\beta(g))=g$ for all $g\in G$,
 			\item $\alpha(\epsilon(x))=\beta(\epsilon(x))=x$ for all $x\in M$.
 		\end{enumerate}
 		
 		\item An inversion map $i \colon  G \to G$, to be denoted simply by $i(g)=g^{-1}$, such that
 		\begin{enumerate}
 			\item $g^{-1}g=\epsilon(\beta(g))$ and $gg^{-1}=\epsilon(\alpha(g))$.
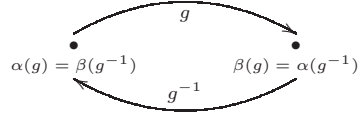
 			$$\xymatrix{*=0{\stackrel{\bullet}{\mbox{\tiny
 							$\alpha(g)=\beta(g^{-1})$}}}{\ar@/^1pc/@<2ex>[rrr]_g}&&&*=0{\stackrel{\bullet}{\mbox{\tiny
 							$\beta(g)=\alpha(g^{-1})$}}}{\ar@/^1pc/@<2ex>[lll]_{g^{-1}}}}$$
 		\end{enumerate}
 	\end{enumerate}

 	A groupoid $G$ over a base $M$ is denoted by $\xymatrix@1@C=1.5em{
 		G\ar@<.3ex>[r]^{\alpha}\ar@<-.3ex>[r]_{\beta}&M
 	}$ or simply $G\rightrightarrows M$.
 	
 	It is easy to see that $\epsilon$ must be injective, so there is a
 	natural  identification between $M$ and $\epsilon(M)$. 
 	
 	\paragraph{Lie groupoids.}
 	
 	A groupoid, $G \rightrightarrows M$, is said to be a Lie
 		groupoid if $G$ and $M$ are differentiable manifolds, all the structural maps
 	are differentiable and besides, $\alpha$ and $\beta$ are differentiable
 	submersions. If $G \rightrightarrows M$ is a Lie groupoid, then $m$
 	is a submersion, $\epsilon$ is an embedding and $\iota$ is a
 	diffeomorphism. Notice that since $\alpha$ and $\beta$ are
 	submersions, the $\alpha$ and $\beta$-fibers are submanifolds. The same properties
 	imply that $G_2$ is a submanifold. We will use $G^x=\alpha^{-1}(x)$,
 	$G_y=\beta^{-1}(y)$ and $G^x_y=\alpha^{-1}(x)\cap\beta^{-1}(y)$.
 	
 	\paragraph{Left and right multiplications.}
 	Given $g\in G_y^x$, so $g:x\to y$, we can define two (bijective) mappings
 	$L_{g}\colon G^y\to G^x$ and $R_{g}\colon G_x\to G_y$, which are the
 	left translation by $g$ and the right translation by $g$,
respectively. These diffeomorphisms are given by
 	\begin{equation}\label{leftrightmultiplication}
 	\begin{array}{lcr}
 	\begin{array}{rccl}
 	L_{g}\colon &G^y&\longrightarrow &G^x \\ 
 	&h &\mapsto&L_{g}(h) = gh
 	\end{array}
 	&; &
 	\begin{array}{rccl}
 	R_{g}\colon &G_x&\longrightarrow& G_y \\ 
 	&h& \mapsto &R_{g}(h) = hg,
 	\end{array}
 	\end{array}
 	\end{equation}
 	where $(L_{g})^{-1} = L_{g^{-1}}$ and $(R_{g})^{-1} = R_{g^{-1}}$.
 	
 	\paragraph{Bisections.}
 	A submanifold $L\subset G$ is called a bisection of $G$ if the
 	restricted maps, $\alpha_{|L}: \ L\rightarrow M$ and $\beta_{|L}: \
 	L\rightarrow M$ are both diffeomorhisms. Consequently, for any
 	bisection $L\subset G$, there is a  corresponding $\alpha$-section
 	$L_\alpha =(\alpha_{|L})^{-1}: \ M\rightarrow L$, where $\beta\circ
 	L_\alpha:\ M\rightarrow M$ is a diffeomorphism. Likewise, there is a
 	$\beta$-section $L_\beta=(\beta_{|L})^{-1}: \ M\rightarrow G$, where
 	$\alpha\circ L_{\beta}=(\beta\circ L_\alpha)^{-1}:\ M \rightarrow M $
 	is a diffeomorphism. More generally, $L\subset G$ is called a local
 	bisection if the restricted maps $\alpha_{|L}$ and $\beta_{|L}$ are
 	local diffeomorphisms onto open sets, $U,\ V\subset M$,
 	respectively. Local bisections on a Lie groupoid always exist.
 	
 	\paragraph{Invariant vector fields.}
 	A vector field $X$ on $G$ is said to be
 	left-invariant (resp., right-invariant) if it is
 	tangent to the fibers of $\alpha$ (resp., $\beta$) and
 	\[X(gh) = (T_{h}L_{g})(X(h))  \quad
 	\Big(\textrm{ resp. }X(gh)= (T_{g}R_{h})(X(g)) \Big)\]for all $(g,h) \in G_{2}$.
 	
 	\paragraph{Morphisms.}
 	Given two Lie groupoids $G \rightrightarrows M$ and $G'
 	\rightrightarrows M'$, a morphism of Lie groupoids is a
 	smooth map $\Phi\colon  G \to G'$ such that
 	\begin{enumerate}
 		\item If $(g, h) \in G_{2} $ then $ (\Phi(g), \Phi(h)) \in (G')_{2}$ and
 		\item $\Phi(gh) = \Phi(g)\Phi(h)$.
 	\end{enumerate}
 	
 	A morphism of Lie groupoids $\Phi\colon  G \to G'$ induces a smooth map
 	$\phi\colon  M \to M'$ in such a way that the source, the
 	target and the identity section commute with the morphism, i.e.
 	\[
 	\alpha' \circ \Phi = \phi \circ \alpha, \makebox[.3cm]{}
 	\beta' \circ \Phi = \phi \circ \beta, \makebox[.3cm]{} \Phi
 	\circ \epsilon = \epsilon' \circ \phi,
 	\]
 	$\alpha$, $\beta$ and $\epsilon$ (resp., $\alpha'$, $\beta'$ and
 	$\epsilon'$) being the source, the target and the identity sections
 	of $G$ (resp., $G'$).

 	\subsection{Lie algebroid associated to a Lie groupoid}\label{App:LieAlgForGroup}

 	Given a Lie groupoid $G$ we denote by
 	$AG=\ker(T\alpha)_{|\epsilon(M)}$, i.e., the set of vectors tangent to
 	the $\alpha$-fibers restricted to the units of the groupoid. Since the
 	units $\im(\epsilon)$ are diffeomorphic to the base manifold $M$, the set $AG$ is understood as a vector bundle $\tau\colon  AG \to M$. The
 	reader should keep this identification in mind, because it is going to be
 	used implicitly in some places ($M\equiv \im(\epsilon)\subset G$).
 	
 	It is easy to prove that there exists a
 	bijection between the space of sections $\Gamma(\tau)$ and the set of
 	left-invariant (resp., right-invariant) vector fields on $G$. If $X$
 	is a section of $\tau\colon  AG \to M$, the corresponding left-invariant
 	(resp., right-invariant) vector field on $G$ will be denoted
 	$\lvec{X}$ (resp., $\rvec{X}$), where
 	\begin{equation}\label{linv}
 	\lvec{X}(g) = (T_{\epsilon(\beta(g))}L_{g})(X(\beta(g))),
 	\end{equation}
 	\begin{equation}\label{rinv}
 	\left(\textrm{resp.,\ } \rvec{X}(g) = -(T_{\epsilon
 		(\alpha(g))}R_{g}\circ i)( X(\alpha(g)))\right),
 	\end{equation}
 	for $g \in G$.

 	Using the above facts, we may introduce a Lie algebroid structure
 	$(\lcf\cdot , \cdot\rcf, \rho)$ on $AG$:
 	\begin{enumerate}
 		
 		\item The anchor map $\rho\colon AG\to TM$ is
 		\[
 		\rho(X)(x) = (T_{\epsilon(x)}\beta)(X(x))
 		\]
 		for $X\in \Gamma(\tau)$ and $x \in M$.
 		
 		\item The Lie bracket on the space of sections   $\Gamma(\tau)$,
 		denoted by $\lcf \cdot, \cdot\rcf$, is defined by
 		\[
 		\lvec{\lcf X, Y\rcf} = [\lvec{X}, \lvec{Y}],
 		\]
 		for $X, Y \in \Gamma(\tau)$.

 	\end{enumerate}

    	Note that
 		\[
 		\rvec{\lcf X, Y\rcf} = -[\rvec{X}, \rvec{Y}], \makebox[.3cm]{}
 		[\rvec{X}, \lvec{Y}] = 0,
 		\]
 		\[
 		Ti\circ \rvec{X}=-\lvec{X}\circ i,\;\;\;\; Ti\circ
 		\lvec{X}=-\rvec{X}\circ i,
 		\]
 		(for more details, see \cite{Mackenzie}). The dual bundle of $AG$ will be denoted by $A^*G$. 

        A Lie groupoid morphism 
        $$\xymatrix{ G \ar[rr]^{\Phi}  \ar[d]<-4pt>  \ar[d]<4pt> && G'\ar[d]<-4pt>  \ar[d]<4pt> \\ M \ar[rr]^{\phi} && M'}
        $$
        induces a Lie algebroid morphism
         $$\xymatrix{ AG \ar[rr]^{A\Phi}  \ar[d]^{\tau}&& AG'\ar[d]^{\tau'} \\ M \ar[rr]^{\phi} && M'}
        $$
        given by 
        $$A\Phi(v_{\epsilon (x)})=T_{\epsilon(x)}\Phi(v_{\epsilon (x)
        }) \in A_{\phi(x)}G'\; , \mbox{   for   } v_{\epsilon(x)}\in A_xG.$$

 	\subsection{Examples of Lie groupoids}\label{App:ExampleLieGroupoids}
 	
 	Next, we will present some examples of Lie groupoids. The
 	corresponding associated Lie algebroid is pointed out in each case.
 	
 	\subsubsection{Lie Groups}\label{App:ExampleLieGroup}
 	
 	Any  Lie group $G$ is a Lie groupoid over
 	$\{ e \}$, the identity element of $G$. 
 	\begin{enumerate}
 		\item The source, $\alpha$, is the constant map $\alpha(g)=e$.
 		\item The target, $\beta$, is the constant map $\beta(g)=e$.
 		\item The identity map is $\epsilon(e)=e$.
 		\item The inversion map is $i(g) =g^{-1}$.
 		\item The multiplication is $m(g,h) =g\cdot h$, for
 		any $g$ and $h$ in $G$.
 	\end{enumerate}
 	
 	\paragraph{Associated Lie algebroid:}
 	The Lie algebroid
 	associated with $G$ is just the  Lie algebra ${\mathfrak g}$ of $G$ in a
 	straightforward way.
 	
 	\subsubsection{The Pair or Banal Groupoid} \label{App:ExamplePairGroup}
 	
 	Let $M$ be a manifold. The
 	product manifold $M \times M$ is a Lie groupoid over $M$ called the pair or banal
 		groupoid.  Its structure mappings are:
 	\begin{enumerate}
 		\item The source, $\alpha$, is the projection onto the first
 		factor.
 		\item The target, $\beta$, is the projection onto the second factor.
 		\item The identity map is $\epsilon(x)
 		= (x, x)$, for all $x \in M$.
 		\item The inversion map is $i(x, y) = (y, x)$.
 		\item The multiplication is $m((x, y), (y, z)) = (x, z)$, for
 		$(x, y), (y, z) \in M \times M$.
 	\end{enumerate}
 	
 	\paragraph{Associated Lie Algebroid:} 
 	If $x$ is a point of $M$, it follows that
 	\[
 	A_x=\ker(T\alpha)_{\epsilon(x)}=\{0_x\}\times T_xM
 	\]
 	which gives the vector bundle structure and given $(0_x,X_x)\in A_x$,
 	then \[\tau (0_x,X_x)=x.\]
 	\begin{enumerate}
 		\item The anchor is given by the projection over the second factor
 		$\rho(0_x,X_x)=X_x\in T_xM$.
 		\item The Lie bracket on the space of sections, $\Gamma(\tau)$, is the
 		Lie bracket of vector fields on the second factor $\lcf (0,X),(0,Y)\rcf=(0,[X,Y])$.
 	\end{enumerate}
 	In conclusion, we have that the Lie algebroid $A(Q \times Q) \to Q$ may be identified 
 	with the standard Lie algebroid $\tau_Q: TQ \to Q$.
 	
 	\subsubsection{Atiyah or Gauge Groupoids}\label{App:Atiyah}
 	
 	Let $\pi\colon  P \rightarrow M$ be a
 	principal $G$-bundle. Then the free action $\Phi\colon  G \times P \to
 	P$ induces the diagonal action  $\Phi'\colon  G \times (P \times
 	P) \to P\times P$ by $\Phi'(g, (q, q')) = (gq, gq')$. Moreover, one may consider the quotient manifold $(P
 	\times P) / G$ and it admits a Lie groupoid structure over $M$,  called the
Atiyah or Gauge groupoid (see, for instance, \cite{Mackenzie,MMM06Grupoides}). We describe now the structural mappings.
 	\begin{enumerate}
 		\item The source, ${\alpha}\colon  (P \times P) / G \to M $ is
 		given by $[(q,
 		q')] \mapsto \pi(q)$.
 		\item The target, ${\beta}\colon  (P \times P) / G \to M$  is given by
 		$[(q,
 		q')] \mapsto \pi(q')$.
 		\item The identity map, ${\epsilon}\colon  M \to (P \times P) / G$ is
 		$x\mapsto [(q, q)], \; \mbox{ if } \pi(q) = x.$
 		\item The inversion map, ${i}\colon  (P \times P) / G \to (P
 		\times P) / G$ is $ [(q, q')] \mapsto [(q', q)]$.
 		\item The multiplication map ${m}\colon  ((P \times P) / G)_{2} \to (P
 		\times P) / G$ is $ ([(q, q')], [(gq', q'')]) \mapsto [(gq, q'')]$.
 	\end{enumerate}
 	
 	\paragraph{Associated Lie Algebroid:}
 	It easily follows that $A=\ker{T\alpha}_{\epsilon(M)}$ can be
 	identified with $TP/G$. Then the associated Lie algebroid is just
 	$\tau: TP/G\rightarrow M$, where $\tau$ is the obvious
 	projection and the Lie algebroid structure is provided by
 	\begin{enumerate}
 		\item The anchor, $\rho:TP/G\rightarrow TM$, is given by the quotient
 		of the natural projection map $T\pi:TP\rightarrow TM$. That is,
 		$\rho=\widetilde{T\pi}:TP/G\rightarrow TM$.
 		\item 
 		The space of sections of the vector bundle $\tau: TP/G \to P/G = M$ may be identified with the set of $G$-invariant vector fields on $P$. Under this identification, the Lie bracket on the space of sections is given by the standard Lie
 		bracket of vector fields. We remark that it is easy to see that the Lie bracket of
 		two $G$-invariant vector fields is another $G$-invariant vector field.
 	\end{enumerate}
 	
 	\subsubsection{Action Lie groupoids}\label{App:action}
 	
 	Let $G$ be a Lie group and let ${\Phi}\colon M\times G\to M$, $(x,g)\mapsto xg$, be a right action of $G$  on $M$.  Consider the action Lie groupoid $M\times G$ over $M$ with
 	structural maps given by
 	\begin{enumerate}
 		\item The source is ${\alpha}(x,g)=x$.
 		\item The target is ${\beta}(x,g)=xg$.
 		\item The identity map is ${\epsilon }(x)=(x,{e})$.
 		\item The inversion map is ${i}(x,g)=(xg, g^{-1})$.
 		\item The multiplication is $m((x,g),(xg,g'))=(x,gg')$
 	\end{enumerate}
 	See, for instance, \cite{Mackenzie,2015JCDavidLocalDiscrete} for the details.
 	
 	\paragraph{Associated Lie Algebroid:}
 	Now, let ${\mathfrak g}=T_{{ e}}G$ be the Lie algebra of $G$. Given
 	$\xi\in\mathfrak{g}$ we will denote by $\xi_M$ the infinitesimal generator of the action
 	$\Phi\colon M\times G\to M$. Consider now the
 	vector bundle $\tau:M\times \mathfrak{g}\rightarrow M$ where $\tau$ is
 	the projection over the first factor, endowed with the following
 	structures:
 	\begin{enumerate}
 		\item The anchor is $\rho(x,\xi)=\xi_M(x)$.
 		\item The Lie bracket on the space of sections is given by $\lcf
 		\widetilde{\xi},\widetilde{\eta}\rcf(x)=[\widetilde{\xi}(x),\widetilde{\eta}(x)]
 		+
 		(\widetilde{\xi}(x))_M(x)(\widetilde{\eta})-(\widetilde{\eta}(x))_M(x)(\widetilde{\xi})$
 		for $\widetilde\xi$, $\widetilde\eta\in \Gamma(\tau)$ and $x\in M$.
 	\end{enumerate}
 
 \section{Symplectic Groupoids}\label{App:SymplGroupoid}
 
We review some results and basic constructions on symplectic groupoids (for more details see~\cite{coste} and~\cite{2021CrainicBook}). 
 \begin{definition}
 	
 	A \textbf{symplectic groupoid} is a Lie groupoid $G\rightrightarrows M$
 	equipped with a symplectic form $\omega$ on $G$ such that the graph of
 	the multiplication $m\ \colon \  G_2\to G$, that is, the set $\left\{ (g,h,gh)\,|\, (g,h)\in G_2 \right\}$, is a Lagrangian submanifold of $G\times G\times G^-$ with the product symplectic form, where $G^-$ denotes $G$ endowed with the symplectic form $-\omega$.
 \end{definition}
 
 \begin{remark}
 	One may prove that $(G, \omega)$ is a symplectic groupoid
 	in the above sense if and only if the $2$-form $\omega$ is multiplicative. We
 	say that the form $\omega$ is multiplicative if and only if 
 	\[
 	m^*\omega=\pi_1^*\omega+\pi_2^*\omega
 	\]
 	where $\pi_i:G_2\rightarrow G$, $i=1,\ 2$ are the projections over the
 	first and second factor.
 \end{remark}
 
 Assume that $G\rightrightarrows M$ is a Lie groupoid with source and
 target $\alpha$ and $\beta$ respectively and identity section,
 inversion map and multiplication $\epsilon$, $i$ and $m$, then there is an induced Lie
 groupoid structure $T^*G\rightrightarrows A^*G$ which we define below
 (see Appendix~\ref{App:lieGroupoid} for notation). The cotangent bundle $T^*G$ turns out to be a symplectic groupoid with the canonical symplectic form $\omega_G$ (see 
 \cite{coste,2021CrainicBook}). Let us define the composition law on $T^*G$, which will be written as
 $\oplus_{T^*G}$, and let $\widetilde \alpha$, $\widetilde \beta$,
 $\widetilde \epsilon$ and $\widetilde i$ stand for the source, target,
 identity section and inversion map of the groupoid structure on $T^*G$
 defined below.  Each one of these maps will cover the corresponding structural map of $G$.
 
 \begin{enumerate}
 	\item The source map is
 	defined in a way that the following diagram is commutative
 	\begin{equation}\label{comm1}
 	\xymatrix{
 		T^*G \ar[r]^{\widetilde\alpha}\ar[d]_{\pi_G}\ar@{}[dr]|\circlearrowright
 		& A^*G\ar[d]^\tau\\
 		G\ar[r]^\alpha & M
 	}
 	\end{equation}
 	Assume that $g\in G^x_y=\alpha^{-1}(x)\cap \beta^{-1}(y)$ and let $X\in A_xG$. So, by the definition of Lie algebroid
 	associated to a Lie groupoid, $X$ is a tangent vector at $\epsilon(x)$ that is tangent
 	to $G^x=\alpha^{-1}(x)$.  Then, since $i\ \colon\ G^x\rightarrow G_x$
 	(where $G_x=\beta^{-1}(x)$) is a
 	diffeomorphism between the $\alpha$-fibers and $\beta$-fibers, we
 	conclude that $-(T_{\epsilon(x)}i)(X)\in T_{\epsilon(X)}G_x$, i.e., it is
 	tangent to the $\beta$-fiber. Now,
 	recalling that right multiplication is a bijection $R_g:G_x\rightarrow
 	G_y$, with $g \in G^x_y$, we have that $T_{\epsilon(x)}R_g(-T_{\epsilon(x)}i(X))\in
 	T_gG$ and we can finally define
 	\begin{equation}\label{eq:ApB_source}
 	\widetilde\alpha(\varpi_g)(X)=\varpi_g(-T_{\epsilon(x)}(R_g\circ i) (X)).
 	\end{equation}
 	In particular, if we are dealing with a section $X\in\Gamma(\tau)$, the previous
 	construction leads us to right-invariant vector fields:
 	\begin{equation}\label{alpha}\widetilde\alpha(\varpi_g)(X)=\varpi_g(\rvec{X}(g)),\text{ for }X\in\Gamma(\tau),\end{equation}
 	where $\tau\ \colon AG \to M$ is the Lie algebroid associated to $G$ (see equation
 	\eqref{rinv} in Appendix~\ref{App:lieGroupoid}). 
 	
 	Notice that $\widetilde \alpha\ \colon\ T^*G\to A^*G$ is a surjective submersion.
 	
 	\item In an analogous way, the target map, $\widetilde \beta\ \colon\ T^*G\to A^*G$, is defined so that the following diagram commutes
 	\begin{equation}\label{comm2}
 	\xymatrix{
 		T^*G \ar[r]^{\widetilde\beta}\ar[d]_{\pi_G}\ar@{}[dr]|\circlearrowright
 		& A^*G\ar[d]^\tau\\
 		G\ar[r]^\beta & M.
 	}
 	\end{equation}
 	Now, given $\varpi_g$ and assuming that $g\in G^x_y$, since left
 	multiplication is defined on the $\alpha$-fibers, $L_g\ \colon G^y\rightarrow
 	G^x$, then $T_{\epsilon(x)}L_g(X)\in T_{g}G$ and we can define
 	\begin{equation}\label{eq:ApB_target}
 	\widetilde\beta(\varpi_g)(X)=\varpi_g(T_{\epsilon(x)}L_g(X)).
 	\end{equation}
 	In other words, 
 	\begin{equation}\label{beta}\widetilde\beta(\varpi_g)(X)=\varpi_g(\lvec{X}(g)),\text{ for
 	}X\in\Gamma(\tau)\end{equation}from equation \eqref{linv}, Appendix~\ref{App:lieGroupoid}.

 	\item The identity map, $\widetilde \epsilon : A^*G\to
 	T^*G$, is defined so that
 	\[
 	\xymatrix{
 		A^*G \ar[r]^{\widetilde\epsilon}\ar[d]_{\tau}\ar@{}[dr]|\circlearrowright
 		& T^*G\ar[d]^{\pi_G}\\
 		M\ar[r]^\epsilon & G
 	}
 	\]
 	commutes. Take $v\in T_{\epsilon(x)}G$. We can obtain an element of $A_xG\subset T_{\epsilon(x)}G$ by computing $v-T(\epsilon\circ\alpha)(v)$. This tangent vector is indeed tangent to an $\alpha$-fiber since $T\alpha$ of it is $0$, using an argument analogous to that in the definition of $\widetilde \alpha$. Then, for $\mu_x\in A^*_xG$, we define
 	\[ \widetilde \epsilon(\mu_x)(v)=\mu_x(v-T(\epsilon\circ\alpha)(v)). \]
 	\item The inversion map, $\widetilde i\colon T^*G\to T^*G$, is defined as a mapping from each $T^*_gG$ to $T^*_{g^{-1}}G$. If $X\in T_{g^{-1}}G$, then $Ti(X)\in T_gG$, so  for $\varpi_g\in T^*_gG$ we can define
 	\[ \widetilde i (\varpi_g)(X)=-\varpi_g(Ti(X)).\]
 	
 	\item The groupoid operation $\oplus_{T^*G}$ is defined for
 	these pairs $(\varpi_g,\nu_h)\in T^*G\times T^*G$ satisfying the
 	composability condition $\widetilde \beta(\varpi_g)=\widetilde
 	\alpha(\nu_h)$, which in particular implies that $(g,h)\in G_2$
 	using diagrams \eqref{comm1} and \eqref{comm2}. This condition can be rewritten as
 	\[ \varpi_g\circ TL_g=-\nu_h\circ TR_h\circ Ti\quad\text{on }A_{\beta(g)}G\subset T_{\epsilon(\beta(g))}G. \]
 	We define 
 	\[
 	\left( \varpi_g\oplus_{T^*G}\nu_h \right)(T_{(g,h)}m(X_g,Y_h))=\varpi_g(X_g)+\nu_h(Y_h),
 	\]
 	for $(X_g, Y_h)\in T_{(g,h)}G_2$. An explicit expression for $\oplus_{T^*G}$ using local bisections in the Lie groupoid $G$ can be found in \cite{2015JCDavidLocalDiscrete}.
 \end{enumerate}
 
 \begin{remark}
 	Note from equation \eqref{alpha} (resp.\ \eqref{beta}) that the definition of
 	$\widetilde\alpha$ (resp.
 	$\widetilde\beta$) is just given by
 	``translation'' via right-invariant vector fields (resp.\ left-invariant).
 \end{remark}
 \begin{remark}
 	Another interesting property is that the application $\pi_G\ \colon\  T^*G\to G$ is a
 	Lie groupoid morphism over the vector bundle projection $\tau_{A^*G}\colon A^*G\to M$.
 \end{remark}
 
 \begin{remark}
 	When the Lie groupoid $G$ is a Lie group, the Lie groupoid $T^*G$ is not in general a Lie group. The base $A^*G$ is identified with the dual of the Lie algebra $\mathfrak{g}^*$, and we have $\widetilde\alpha(\varpi_g)(\xi)=\varpi_g\left( TR_g \xi\right)$ and $\widetilde\beta(\varpi_g)(\xi)=\varpi_g\left( TL_g \xi\right)$, where $\varpi_g\in T^*_gG$ and $\xi\in {\mathfrak g}$.
 \end{remark}

 \subsection{Properties}\label{App:PropertiesSymplGroup}
 The following theorem will be crucial in the paper (see
 \cite{coste,2021CrainicBook} for a proof, as well as the references therein):
 
 \begin{theorem}\label{theorem-sg}
 	Let $G\rightrightarrows M$ be a symplectic groupoid, with symplectic 2-form $\omega$. We have the following properties:
 	\begin{enumerate}
 		\item For any point $g\in G_y^x$, the subspaces $T_g G_y$ and $T_gG^x$ of the symplectic vector space $(T_gG, \omega_g)$ are mutually symplectic orthogonal. That is,
 		\begin{equation}\label{dualpair}
 		T_g G_y=(T_gG^x)^{\perp}.
 		\end{equation}
 		\item The submanifold $\epsilon (M)$ is a Lagrangian submanifold of the symplectic manifold $(G, \omega)$.
 		\item The inversion map $i\ \colon\  G\to G$ is an anti-symplectomorphism of $(G, \omega)$, that is, $i^*\omega=-\omega$.
 		\item There exists a unique Poisson structure $\Pi$ on $M$ for which $\beta\ \colon\  G\to M$ is a Poisson map, and $\alpha\ \colon\  G\to M$ is an anti-Poisson map (that is, $\alpha$ is a Poisson map when $M$ is equipped with the Poisson structure $-\Pi$).
 	\end{enumerate}
 \end{theorem}
 
 \begin{remark} The theorem above states that the $\alpha$-fibers and
 	$\beta$-fibers are symplectically orthogonal. A pair of fibrations
 	satisfying that property are called a dual pair. This dual
 	pair property will be the keystone of our construction.
 \end{remark}
 
 \begin{remark}\label{remark:AppB} When dealing with the symplectic groupoid $T^*G$, where
 	$G$ is a groupoid, the Poisson structure on $A^*G$ is the
 	(linear) Poisson structure $\Lambda_{A^*G}$ on the dual vector bundle of $AG$ (see
 	Appendix~\ref{App:lieGroupoid}).
 \end{remark}

 \subsection{The Lie algebroid associated to a cotangent groupoid}\label{App:LieAlgCotangGroup}
 The Lie algebroid associated to the cotangent groupoid is precisely
 $\pi_{A^*G}\colon T^*(A^*G)\rightarrow A^*G$ and 
 a Lie algebroid isomorphism 
 $$\phi^{\omega_G}\colon A(T^*G)=V_{\widetilde{\epsilon}(A^*G)}\widetilde{\alpha}\rightarrow T^*(A^*G)$$
 is characterized by the following condition
 $$i_{X_{\widetilde{\epsilon}(\mu_x)}}\omega_G(\widetilde{\epsilon}(\mu_x))=-\left(T^*_{\widetilde{\epsilon}(\mu_x)} \widetilde{\beta}\right) (\phi^{\omega_G}(X_{\widetilde{\epsilon}(\mu_x)})), $$
 for $\mu_x\in A_x^*G$ and $X_{\widetilde{\epsilon}(\mu_x)}\in V_{\widetilde{\epsilon}(\mu_x)}\widetilde{\alpha}=A_{\mu_x}(T^*G).$

 Moreover, under the previous identification, the anchor map of $T^*(A^*G)\rightarrow A^*G$ is just the vector bundle morphism
 $$\Lambda^\sharp_{A^*G}\colon T^*(A^*G)\rightarrow T(A^*G),$$
 where $\Lambda_{A^*G}$ is the linear Poisson structure on $A^*G$ induced by the Lie algebroid structure on $AG$ (see Appendix~\ref{App:lieGroupoid}).

 \end{appendices}

\section*{Acknowledgements}

The authors acknowledge financial support from Grants PID2022-137909-NB-C21, PID2022-137909-NB-C22,  PCI2024-155047-2, TED2021-129455B-I00 and   CEX2023-001347-S funded by MICIU/AEI/10.13039/501100011033.

\bibliography{References}

\end{document}